%% file: main.tex
\documentclass{article}

\usepackage{environment-setup}
\usepackage{macros}

\title{On the gap of quiver representations}
\author{
John Maar
\thanks{Technical University Berlin. \ \email{maar@math.tu-berlin.de}}
}
\date{\today}

\begin{document}

\maketitle

\begin{abstract}
The \emph{nullcone membership problem}, deciding whether an orbit closure contains the origin, is fundamental in computational invariant theory. For self‑adjoint groups, Bürgisser, Franks, Garg, Oliveira, Walter and Wigderson  gave a geodesic optimization algorithm whose complexity is controlled by the \emph{gap}, a condition number of the representation. We study the gap for quiver representations under the action of the special linear group.

We prove that the inverse gap is polynomially bounded in the number of vertices and the maximum dimension for type $A$ and $\hat{A}$ quivers, as well as tree quivers with uniform dimension vectors. Consequently, the algorithm of Bürgisser et al. solves the nullcone membership problem in polynomial time for these families. In contrast, we construct families of quivers and dimension vectors where the gap is exponentially small in the number of leaves, furthermore, for every connected quiver we exhibit dimension vectors such that the \emph{weight margin} (a related condition number) is exponentially small in the number of vertices.

We also extend our results to $\sigma$-semistability, thereby giving a new proof of a recent result of Iwamasa, Oki, and Soma.
    
    
\end{abstract}

\tableofcontents

\section{Introduction}\label{section: introduction}
\input{introduction}

\section{Preliminaries}\label{section: preliminaries}
\input{preliminaries}

\section{Gap of tree quivers with uniform dimension vectors} \label{section: Gap_of_tree_quivers}
\input{tree_quivers}

 \section{Gap of $A_n$ and $\hat{A}_n$ type quivers}\label{section: A_and_hat_A}
\input{A_and_extended_A}

\section{Gap for $\sigma$-semistability}\label{section: sigma_semistability}
\input{sigma_semistability}

\section{Gap of tree quivers with arbitrary dimension vectors} \label{section: star_quivers}
\input{weightmargin_and_gap_star_quivers}

\section*{Acknowledgements}

The author is grateful to his advisor Peter Bürgisser, for many helpful discussions.
Part of this work took place while the author was supported by the Simons Institute for the Theory of Computing, and conducted when the author was visiting the Institute.
The author would like to thank the generosity, hospitality, and excellent research environment provided by the Simons Institute.
The author is supported by the Deutsche Forschungsgemeinschaft (DFG, German Research Foundation, 556164098).

\bibliographystyle{alpha}
\bibliography{References}

 \newpage
\appendix

\makeatletter
\renewcommand{\theHlemma}{\thesection.\arabic{lemma}}
\makeatother

\input{real_minimization_problems}

\end{document}

%% file: introduction.tex
Consider a group $G$ acting linearly on the complex Euclidean space $V=\bC^d$. This action partitions $V$ into orbits $G\cdot v := \{ g\cdot v \;|\; g\in G \}$ for $v\in V$. A fundamental algorithmic question in this setting is the orbit equality problem: given vectors $v,w\in V$ decide if $G\cdot v = G\cdot w$.

This framework captures many classical isomorphism problems. For instance, the module isomorphism problem for finite dimensional algebras is equivalent to an orbit equality problem. Another example is the graph isomorphism problem: Two graphs $G,H$ are isomorphic if and only if their adjacency matrices lie in the same orbit under the symmetric group $\Sym_n$. Here $\sigma\in \Sym_n$ acts on $\bC^n \otimes \bC^n$ via $\sigma \cdot (v\otimes w) = P_\sigma v \otimes P_\sigma w$, where $P_\sigma$ is the associated permutation matrix.

Since $V$ carries the euclidean topology, we may also consider orbit closures. This leads to several algorithmic questions:

\begin{problem}
    Let $G$ act on $V$ and let $u,v\in \bQ(i)^d \subset V$ be vectors with rational coordinates and bit length at most $B\in \bN$. Consider the following problems:
    \begin{enumerate}
    \itemsep0em
    \item \textbf{Orbit equality:} Decide if $G\cdot u = G\cdot v$.
    \item \textbf{Orbit closure containment:} Decide if $v\in \overline{G\cdot u}$.
    \item \textbf{Orbit closure intersection:} Decide if $\overline{G\cdot u} \cap \overline{G\cdot v} \neq \emptyset$.
    \item \textbf{Nullcone membership:} Decide if $0\in \overline{G\cdot v}$.
    \item \textbf{Capacity approximation:} Approximate $\cp(v):= \inf_{g\in G} \|g\cdot v\|$ up to some $\varepsilon > 0$.
\end{enumerate}
\end{problem}
In this paper we restrict ourselves to Zariski closed subgroups $G\subseteq \GL_n$, such that $g\in G$ implies $g^*\in G$. Such subgroups are called self-adjoint and they are exactly the complex reductive linear algebraic groups given in some concrete embedding \cite[\S 3.1-3.2]{Walach}. We also assume that all group actions are rational. Examples falling under these restrictions include the following group actions: \begin{enumerate}
    \itemsep0em
    \item \textbf{simultaneous conjugation:} $\GL_d$ acting on $(\bC^{d\times d})^n$ via $g\cdot (X_1,\ldots, X_n) = (gX_1g^{-1}, \ldots, gX_ng^{-1})$.
    \item \textbf{left-right action:} $\SL_{d_1}\times \SL_{d_2}$ acting on $(\bC^{d_1\times d_2})^n$ via $(g,h)\cdot (X_1,\ldots, X_n) = (gX_1h^{-1}, \ldots, gX_nh^{-1})$.
    \item \textbf{3-tensor action:} $(\GL_d)^3$ acting on $(\bC^{d})^{\otimes 3}$ via $(g_1,g_2,g_3) \cdot (v_1 \otimes v_2\otimes v_3) = (g_1v_1 \otimes g_2v_2\otimes g_3v_3)$.
    \item \textbf{graph isomorphism:} $\Sym_d$ acting on $\bC^d\otimes \bC^d$ via $g\cdot (v_1\otimes v_2) = gv_1 \otimes gv_2$.
\end{enumerate}

In recent years, the complexity of orbit problems has been extensively studied. For instance, the orbit closure intersection problem admits polynomial time algorithms for the left-right action \cite{left_right_numerical, OCS, IcanyosQiao2023} and the conjugation action \cite{OCS, ForbesShpilka2013}. When $G$ is commutative, there are polynomial time algorithms for all of the above problems \cite{https://doi.org/10.4230/lipics.ccc.2021.32}. On the other hand, orbit closure containment for three tensors is NP-hard \cite{OCinclusion_tensors} and orbit closure intersection is graph-isomorphism-hard for various tensor-tuple-representations \cite{lysikov2024complexitytheoryorbitclosure}.

In this paper we are concerned with representations arising from quivers, more precisely for a quiver $Q$ and dimension vector $\alpha$ we consider the representation space $\Rep_\alpha(Q) := \bigoplus_{a\in Q_1} \bC^{\alpha(ha)\times \alpha(ta)}$ with the natural $\GL_\alpha := \prod_{v \in Q_0} \GL_{\alpha(v)}$ and $\SL_\alpha := \prod_{v \in Q_0} \SL_{\alpha(v)}$ actions.

In the case of the $\GL_\alpha$ action all of the above problems, except orbit closure containment, admit polynomial time algorithms \cite{Burgisser_2019, modIso}, while the $\SL_\alpha$ action appears much harder, which stems from the fact, that the invariant theory of the $\SL_\alpha$ action is more complicated than that of the $\GL_\alpha$ action. One general result is that the orbit equality problem for the $\SL_\alpha$ action can be reduced to the module isomorphism problem which admits a polynomial time algorithm \cite{modIso}.

When approaching the orbit closure intersection problem, there are generally two approaches. The first is invariant theoretic: By a remarkable theorem of Mumford \cite{GeometricInvariantTheory} two orbit closures $\overline{G\cdot v},\overline{G\cdot w}$ intersect if and only if every invariant polynomial takes the same value on $v$ and $w$. This naturally leads to the notion of a separating set of invariants, i.e. a set of invariants, that separates orbit closures in this way. In case the set is small, i.e. polynomial in the input parameters and the invariants can be computed efficiently, this yields an efficient algorithm for the orbit closure intersection problem. This is the approach taken in \cite{OCS} for the simultaneous conjugation and left-right action. The second approach is based on numerical optimization algorithms and aims to decide orbit closure intersection, by checking if the distance of the two orbit closures falls under a certain threshold. This is the approach taken in \cite{left_right_numerical} for the left-right and simultaneous conjugation actions.


\subsection{Geodesic optimization and gap}

Consider a connected, self-adjoint subgroup $G \subseteq \GL_n$ acting on a finite dimensional vector space $V$ by linear transformations $\pi:G\to \GL(V)$ as discussed in \cite{Burgisser_2019}. Here, \emph{self-adjoint} means that $G$ is Zariski closed and satisfies $g\in G \implies g^*\in G$. For such a group, we denote $K:= G\cap U_n$, the set of unitary matrices in $G$, which form a maximal compact subgroup of $G$. Associated with $G$ and $K$ are the Lie-algebras
\[\Lie(K) := \{H \in \bC^{n\times n} \;|\; \forall t\in \bR : e^{tH} \in K\} \quad \Lie(G) := \{H \in \bC^{n\times n} \;|\; \forall t\in \bR : e^{tH} \in G\}\]
Then $\Lie(K)$ is a real Lie-algebra and $\Lie(G)$ is a complex Lie-algebra. Moreover $\Lie(K) \subseteq \Lie(U_n) = i \Herm_n$, the space of skew-hermitian matrices. The representation $\pi$ induces a Lie-algebra representation
\[ \Pi: \Lie(G) \to \Lie(\GL(V)) = \LL(V),\quad  H \mapsto \partial_{t=0} \pi(e^{tH}).\]
The representation $\pi$ also induces an action on the polynomial ring $\bC[V]$ via $(g\cdot f)(v) := f(g^{-1} \cdot v)$. The invariant polynomials form a ring called the \emph{invariant ring}
\[ \bC[V]^G := \{ f\in \bC[V] \;|\; g\cdot f = f \}.\]
A vector $v\in V$ is called \emph{semistable}, when its orbit closure does not contain the origin, otherwise it is called \emph{unstable}. The \emph{capacity} of a vector $v\in V$ is defined as
\[ \cp(v) := \inf_{g\in G} \|g\cdot v\|. \]
The \emph{nullcone} is defined as the set of unstable vectors $v$. It is an important result by Hilbert, that the nullcone is also the vanishing locus of the non-constant homogeneous invariants of $\pi$, see for example \cite[\S 9]{ITQR}.
\begin{equation}
    \label{eqn_nullcone_def}
    \NC(\pi) :=\{ v\in V \;|\; 0\in \overline{G\cdot v} \} = \{v\in V \;|\; f(v) = 0 \text{ for all } f\in \bigoplus_{d=1}^\infty \bC[V]^G_{d} \}
\end{equation}
In \cite{Burgisser_2019} the authors proposed the framework of geodesic optimization to solve orbit problems numerically. The runtime of these algorithms depends on certain condition numbers, which depend on the group action and might be exponentially small in the input size. An important tool in the analysis of these algorithms is the so called moment map, which can be thought of as a derivative of the \emph{Kempf-Ness function}
\[ F_v : G\to \bR, \quad g\mapsto \frac{1}{2}\log \|\pi(g) \cdot v\|^2.\]
The \emph{moment map} $\mu:V\backslash\{0\} \to i \Lie(K)$ is defined by the property that for all $H\in i\Lie(K)$, we have
\[ \tr(\mu(v) H) = \partial_{t=0} F_v(e^{tH}) = \frac{\langle v, \Pi(H)v\rangle}{\|v\|^2}.\]
Now we are ready to define the key condition numbers. The \emph{gap} of the representation $\pi$ is defined as
\[ \gamma(\pi) := \inf\{ \|\mu(v)\| \;|\; v\in \NC(\pi)\backslash\{0\} \}. \]
Fix a maximal commutative subgroup $T$ of $G$ (For $\SL_\alpha$ the natural choice is $\ST_\alpha$). Restricting $\pi$ to $T$ gives a representation $\pi_T :T \to \GL(V)$. Its gap is called the \emph{weight margin} $\gamma_T(\pi) := \gamma(\pi_T)$ and is independent of the choice of torus, as follows from the characterization below. The weight margin plays a central role in the algorithms of \cite{Burgisser_2019}. For any $G$ and $\pi$, the gap is lower bounded by the weight margin
\[\gamma_T(\pi) \leq \gamma(\pi).\]
A \emph{weight} of the representation $\pi_{T}$ is a vector $w \in \bZ^n$, such that there exists a vector $v \in V$ such that
\[ \diag(t_1, \ldots, t_n) \cdot v = \prod_{k=1}^n t_k^{w_k} v, \]
for all $t\in T$. The set of weights is denoted as $\Omega(\pi)$ and independent of the choice of torus $T$. The weight margin can be expressed in terms of weights as
\[ \gamma_T(\pi) = \min \{\dist(\conv(A), 0) \; |\; A \subseteq \Omega(\pi),\; 0 \not\in \conv(A)\},\]
where $\dist$ is the euclidean distance. For proofs of these statements we refer the reader to \cite{Burgisser_2019}. The \emph{weight norm} of $\pi$ is defined as
\[ N(\pi) := \max \{ \|w\|_2 \; |\; w \in \Omega(\pi) \}.\]
The importance of the gap for the nullcone membership problem is evident from Algorithm 4.2, Theorem 4.2 and Theorem 8.4  in \cite{Burgisser_2019}, which we restate here

\begin{algorithm}
\label{algorithm_4.2}
\caption{Algorithm 4.2 in \cite{Burgisser_2019} for the scaling problem}
\KwIn{\begin{itemize}
    \item A vector $v \in V$.
    \item Oracle access to the moment map restricted to the orbit of $v$, i.e. $g \mapsto \mu(g \cdot v)$.
    \item A number of iterations $T$ and a step size $\eta$.
\end{itemize}}
\KwOut{A group element $g \in G$.}
$g_0 \gets I$\;
\For{$t \gets 0$ \KwTo $T-1$}{
$g_{t+1} \gets e^{-\eta\mu(g_t \cdot v)}g_t$\;}
\KwRet{$g_T$}
\end{algorithm}

\begin{theorem}[Theorem 4.2 in \cite{Burgisser_2019} ]
    Let $v \in V$ be a semistable vector (i.e. $\cp(v) > 0$). For every $\varepsilon > 0$, Algorithm \ref{algorithm_4.2}  with step size $\eta := \frac{1}{2 N(\pi)^2}$ and
    \[ T \geq \frac{4 N(\pi)^2}{\varepsilon^2} \log(\frac{\|v\|}{\cp(v)}) \]
    iterations returns a group element $g\in G$ such that $\|\mu(g \cdot v)\| \leq \varepsilon$.
\end{theorem}

\begin{theorem}[Theorem 8.4 in \cite{Burgisser_2019}]
\label{theorem_runtime8.4}
    Let $G := \SL(n_1) \times \ldots \times \SL(n_k)$ and $\pi : G \to \GL(m)$ be the restriction of a polynomial representation with degrees bounded by $d$ and bitsize of coefficients bounded by $R$ assume further that $K^{-1} \|v\| \leq \|v\|_2 \leq K \|v\|$ for any $v \in \bC^m$. Put $n:= n_1 +\ldots+n_k$. Let $v \in V$ be a vector, whose components are gaussian integers bounded in absolute value by $M$. Then Algorithm \ref{algorithm_4.2} with $\varepsilon := \frac{\gamma(\pi)}{2}$ solves the nullcone membership problem in time
    \[ \poly(K, n, d, \log(RmMK), \gamma(\pi)^{-1}). \]
\end{theorem}

\subsection{Representation theory of quivers}

Throughout this paper we work with the field of complex numbers. In particular all finite dimensional vector spaces are over $\bC$ and endowed with an inner product $\langle v,w\rangle$. A \emph{quiver} $Q=(Q_0,Q_1,h,t)$ consists of a vertex set $Q_0$, and arrow set $Q_1$ and head/tail maps $h,t:Q_1\to Q_0$ indicating the heads and tails of the arrows. 

A \emph{representation} $V$ of $Q$ consists of a collection of finite dimensional vector spaces $V(i)$ for $i\in Q_0$ as well as a collection of maps $V(a):V(ta) \to V(ha)$ for $a\in Q_1$. The vector of dimensions of $V$ is called the \emph{dimension vector} and denoted by $\dim(V):=(\dim(V(i))_{i\in Q_0}$. A \emph{morphism} $\phi:V\to W$ of representations consists of a collection of maps $\phi(i):V(i)\to W(i)$, such that the following diagram commutes for all $a\in Q_1$.
\[\begin{tikzcd}
	{V(ta)} \arrow[r, "V(a)"] \arrow[d, "\phi(ta)"] & {V(ha)} \arrow[d, "\phi(ha)"] \\
	{W(ta)} \arrow[r, "W(a)"] & {W(ha)}
\end{tikzcd}\]
This yields the \emph{category of representations} of $Q$ denoted by $\Rep(Q)$. The direct sum of two representations $V,W\in \Rep(Q)$ is defined pointwise and arrowwise, i.e. $(V\oplus W)(i) := V(i)\oplus W(i)$ for $i\in Q_0$ and \[(V\oplus W)(a):= \begin{pmatrix}
    V(a) & 0 \\ 0 & W(a)
\end{pmatrix}\]
With this notion of direct sum $\Rep(Q)$ becomes an abelian category. A representation which is isomorphic to a direct sum of two nonzero representations is called \emph{decomposable}. A quiver $Q$ is of \emph{finite representation type}, when there exist only finitely many indecomposable representations of $Q$ up to isomorphism. A quiver $Q$ is of \emph{tame representation type}, when for each dimension vector $\alpha$, one can parametrize all isomorphism classes of indecomposable representations of dimension $\alpha$ (except a finite number) by finitely many 1-parameter families, for the precise definition we refer to \cite[Definition 19.3.3]{ERTAA}. All other quivers are of \emph{wild representation type}. A celebrated result by Gabriel classifies the quivers of finite type as unions of the following classes of quivers, called \emph{Dynkin quivers}, where the subscript denotes the number of vertices and each edge stands for an arrow of arbitrary direction.

\[\begin{tikzcd}
	{A_n:} & \bullet & \bullet & \ldots & \bullet \\
	&&&&& \bullet \\
	{D_n:} & \bullet & \bullet & \ldots & \bullet \\
	&&& \bullet && \bullet \\
	{E_6:} & \bullet & \bullet & \bullet & \bullet & \bullet \\
	&&& \bullet \\
	{E_7:} & \bullet & \bullet & \bullet & \bullet & \bullet & \bullet \\
	&&& \bullet \\
	{E_8:} & \bullet & \bullet & \bullet & \bullet & \bullet & \bullet & \bullet
	\arrow[no head, from=5-2, to=5-3]
	\arrow[no head, from=5-3, to=5-4]
	\arrow[no head, from=5-4, to=5-5]
	\arrow[no head, from=5-4, to=4-4]
	\arrow[no head, from=5-5, to=5-6]
	\arrow[no head, from=7-2, to=7-3]
	\arrow[no head, from=7-3, to=7-4]
	\arrow[no head, from=7-4, to=7-5]
	\arrow[no head, from=7-5, to=7-6]
	\arrow[no head, from=7-6, to=7-7]
	\arrow[no head, from=7-4, to=6-4]
	\arrow[no head, from=9-2, to=9-3]
	\arrow[no head, from=9-3, to=9-4]
	\arrow[no head, from=9-4, to=9-5]
	\arrow[no head, from=9-5, to=9-6]
	\arrow[no head, from=9-6, to=9-7]
	\arrow[no head, from=9-7, to=9-8]
	\arrow[no head, from=9-4, to=8-4]
	\arrow[no head, from=3-2, to=3-3]
	\arrow[no head, from=3-3, to=3-4]
	\arrow[no head, from=3-4, to=3-5]
	\arrow[no head, from=3-5, to=2-6]
	\arrow[no head, from=3-5, to=4-6]
	\arrow[no head, from=1-2, to=1-3]
	\arrow[no head, from=1-3, to=1-4]
	\arrow[no head, from=1-4, to=1-5]
\end{tikzcd}\]

\begin{theorem}[Gabriel's Theorem \cite{Gabriel1972}]
\label{gabriels_theorem}
    A connected quiver $Q$ is of finite type if and only if it's a Dynkin quiver of type $A_n, D_n, E_6, E_7$ or $E_8$. 
\end{theorem}

Gabriel’s theorem was soon extended to the tame case \cite{donovan1973representation, LANazarova_1973}, where the representation theory is still manageable. The connected quivers of tame type are precisely the \emph{extended Dynkin quivers}:

\[\begin{tikzcd}
	{\hat{A}_n:} & \bullet & \bullet & \ldots & \bullet & \bullet \\
	& \bullet &&&& \bullet \\
	{\hat{D}_n:} && \bullet & \ldots & \bullet \\
	& \bullet && \bullet & \bullet & \bullet \\
	{\hat{E}_6:} & \bullet & \bullet & \bullet & \bullet & \bullet \\
	&&&& \bullet \\
	{\hat{E}_7:} & \bullet & \bullet & \bullet & \bullet & \bullet & \bullet & \bullet \\
	&&& \bullet \\
	{\hat{E}_8:} & \bullet & \bullet & \bullet & \bullet & \bullet & \bullet & \bullet & \bullet
	\arrow[no head, from=5-2, to=5-3]
	\arrow[no head, from=5-3, to=5-4]
	\arrow[no head, from=5-4, to=5-5]
	\arrow[no head, from=5-4, to=4-4]
	\arrow[no head, from=5-5, to=5-6]
	\arrow[no head, from=7-2, to=7-3]
	\arrow[no head, from=7-3, to=7-4]
	\arrow[no head, from=7-4, to=7-5]
	\arrow[no head, from=7-5, to=7-6]
	\arrow[no head, from=7-6, to=7-7]
	\arrow[no head, from=9-2, to=9-3]
	\arrow[no head, from=9-3, to=9-4]
	\arrow[no head, from=9-4, to=9-5]
	\arrow[no head, from=9-5, to=9-6]
	\arrow[no head, from=9-6, to=9-7]
	\arrow[no head, from=9-7, to=9-8]
	\arrow[no head, from=9-4, to=8-4]
	\arrow[no head, from=3-3, to=3-4]
	\arrow[no head, from=3-4, to=3-5]
	\arrow[no head, from=3-5, to=2-6]
	\arrow[no head, from=3-5, to=4-6]
	\arrow[no head, from=1-2, to=1-3]
	\arrow[no head, from=1-3, to=1-4]
	\arrow[no head, from=1-4, to=1-5]
	\arrow[no head, from=3-3, to=2-2]
	\arrow[no head, from=3-3, to=4-2]
	\arrow[no head, from=4-4, to=4-5]
	\arrow[no head, from=7-5, to=6-5]
	\arrow[no head, from=7-7, to=7-8]
	\arrow[no head, from=9-8, to=9-9]
	\arrow[no head, from=1-5, to=1-6]
	\arrow[no head, from=1-6, to=1-2, bend right=20]
\end{tikzcd}\]

\begin{theorem}
    A connected quiver $Q$ is of tame type if and only if it is one of the extended Dynkin quivers $\hat{A}_n, \hat{D}_n, \hat{E}_6, \hat{E}_7$ or $\hat{E}_8$
\end{theorem}

We state here a structural result on the dimension vectors of indecomposable representations of type $A$ quivers, which we will use later. This follows from the characterization of indecomposable representations of quivers of finite type which can be found in \cite{Gabriel1972} or any standard text on quiver representations. 

\begin{lemma}
\label{indecomposables_of_AD}
    Let $Q$ be a quiver of type $A$, then any indecomposable representation $V$ has dimension vector with entries bounded by $1$, i.e., $\dim(V)_{\max} = 1$.
\end{lemma}

For a fixed dimension vector $\alpha: Q_0 \to \bN$, we consider the \emph{representation space} of $Q$ with dimension $\alpha$
\[ \Rep_\alpha(Q) := \prod_{a\in Q_1} \bC^{\alpha(ha)\times \alpha(ta)}.\]
This representation space is endowed with a natural group action of $ \GL_\alpha:= \prod_{v\in Q_0} \GL_{\alpha(v)}$ given by
\[ g \cdot X := (g_{ha} X_a g_{ta}^{-1})_{a\in Q_1}. \]
We endow this space with the standard inner product and its induced norm
\[ \langle X, Y\rangle = \sum_{a \in Q_1} \tr(X_a^* Y_a), \quad \|X\|^2 = \sum_{a \in Q_1} \|X_a\|_F^2.\]
The Frobenius norm is the only matrix norm we will use throughout this paper, so we will omit the subscript $F$ from now on. After choosing bases, any representation $V$ of $Q$ can be identified with a point $X\in \Rep_{\dim(V)}(Q)$ and two representations $V,W$ are isomorphic if and only if the corresponding points in the representation space lie in the same orbit. Thus isomorphism classes of $\alpha$ dimensional representations are in one to one correspondence with orbits in the representation space $\Rep_\alpha(Q)$. 
We will usually work with the restriction of the above group action to $\SL_\alpha := \prod_{v\in Q_0} \SL_{\alpha(v)}$ and we denote the gap of this action as $\gamma_{\SL}(Q,\alpha)$, the weight margin as $\gamma_{\ST}(Q,\alpha)$ the null cone as $\NC(Q, \alpha)$.\\

The weight margin of the $\GL_\alpha$ action on $\Rep_\alpha(Q)$ has been discussed in \cite{Burgisser_2019} and they also found an exponential lower bound on the weight margin of the $\SL_\alpha$ action. We restate their theorem in the language of this paper.

\begin{theorem} [Theorem 6.21 in \cite{Burgisser_2019} (partially stated)]
\label{bürgisser_weightmargin_bound}
    Let $Q$ be a quiver with $n$ vertices and $\alpha$ a dimension vector. Then the weight margin of the $\GL_\alpha$ and $\SL_\alpha$ actions on $\Rep_\alpha(Q)$ are lower bounded by
    \begin{align*}
        \gamma_T(Q,\alpha) &\geq n^{-\frac{3}{2}}\\
        \gamma_{\ST}(Q, \alpha) &\geq P^{-n} (M+1)^{-n} M^{-\frac{3}{2}}, 
    \end{align*}
    where $P:= \prod_{k=1}^n \alpha(k)$ and $M:= \sum_{k=1}^n \alpha(k)$.
\end{theorem}

In \cite[Theorem 4.25]{FR21} Franks and Reichenbach constructed a quiver $Q$ with $n$ vertices, $n(n-1)$ arrows and uniform dimension vector $\alpha = (d, \ldots, d)$ such that the gap is exponentially small in the number of vertices
\[ \gamma_{\SL}(Q, \alpha) \leq (d-1)^{-n+1}. \]
Remarkably, the unstable representation they constructed is isomorphic (up to trivial summands) to the representation Derksen and Makam used to prove an exponential degree lower bound for generators of the ring of semiinvariants of quivers \cite[Proposition 1.5]{derksen2016degreeboundssemiinvariantrings}.

Franks and Reichenbach also posed the question, whether polynomial bounds for the gap hold, when the quiver is of finite representation type. We partially answer this question in the affirmative. Our main result is that the inverse of the gap is polynomial in the number of vertices and the maximum entry of the dimension vector $\alpha_{\max}$ for the $\SL_\alpha$ action on $\Rep_\alpha(Q)$, when $Q$ is of type $A$ or $\hat{A}$.

\begin{theorem}[\Cref{mainresult} simplified]
\label{main_simplified}
    Let $Q$ be a quiver of type $A_{n}$ or $\hat{A}_n$ with dimension vector $\alpha$. Then the gap of the $\SL_d$ representation $\Rep_d(Q)$ is lower bounded by 
    \[ \gamma_{\SL}(Q, \alpha) \geq \poly(n^{-1}, \alpha_{\max}^{-1}). \]
\end{theorem} 

This implies the following theorem.
\begin{theorem}
\label{mainresult_finite_type}
    Let $Q$ be a quiver of type $A_{n}$ or $\hat{A}_n$ and $\alpha \in \bN^{Q_0}$ a dimension vector. Let $X \in \Rep_\alpha(Q)$ be a representation whose entries are gaussian rationals with bit length bounded by $B$. Then Algorithm 4.2 in \cite{Burgisser_2019}, solves the nullcone membership problem for $X$ under the $\SL_\alpha$ action on $\Rep_\alpha(Q)$ in time $\mathcal{O}(\poly(\alpha_{\max}, n, B))$, independent of the orientation of $Q$.
\begin{proof}
    After multiplying $X$ with the least common multiple of the denominators of its entries, we can assume that $X$ has entries which are gaussian integers of absolute value bounded by $2^{\alpha_{\max}^2n B}$. We specify \Cref{theorem_runtime8.4} to the group action of $\SL_\alpha = \SL_{\alpha(1)} \times \ldots \times \SL_{\alpha(n)}$ on $\Rep_\alpha(Q)$. The degree of this action is bounded by $\alpha_{\max}$ and the absolute value of the coefficients are bounded by $1$. Thus \Cref{theorem_runtime8.4} yields the result in time
    \[ \poly(n, \sum_k \alpha(k), \alpha_{\max}, \log(n\alpha_{\max}^2 2^{\alpha_{\max}^2n B} ), \gamma_{\SL}(Q, \alpha)^{-1}) \leq \poly(n, \alpha_{\max}, B), \]
    where we used $\gamma_{\SL}(Q, \alpha)^{-1} \leq \poly(n, \alpha_{\max})$ from \Cref{main_simplified}.
\end{proof}
\end{theorem}

\subsection{Our contribution, outline of the paper and open questions}

Theorem 8.2 in \cite{Burgisser_2019} implies, that Algorithm 4.2. in \cite{Burgisser_2019} solves the nullcone membership problem in time polynomial in the input size and the inverse of the gap of the representation. In section 3 we show that the inverse of the gap $\gamma_{\SL}(Q,\alpha)^{-1}$ is polynomial in $d$ and the number of vertices of $Q$ when $Q$ is a tree quiver and $\alpha=(d,\ldots, d)$ is a uniform dimension vector.

\begin{theorem}[\Cref{mainresult_treequivers} restated]
    Let $Q$ be a tree quiver with $n+1$ vertices and uniform dimension vector $\alpha = (d,\ldots, d)$. Then the gap of the $\SL_\alpha$ representation $\Rep_\alpha(Q)$ is lower bounded by 
    \[ \gamma_{\SL}(Q, \alpha) \geq \sqrt{\frac{6}{d(d-1) n(n+1)(2n+1)}} \geq \Omega(n^{-1.5}d^{-1}). \]
\end{theorem} 

In section 4 we show that the inverse of the gap $\gamma_{\SL}(Q,\alpha)^{-1}$ is polynomial in $\alpha_{\max}$ and the number of vertices of $Q$, when $Q$ is a quiver of type $A$ or a quiver of type $\hat{A}$.
\begin{theorem} [\Cref{mainresult} simplified]
    Let $Q$ be a quiver of type $A$ or $\hat{A}$ and $\alpha$ a dimension vector, then the gap of the $\SL_\alpha$ action on $\Rep_\alpha(Q)$ is lower bounded by
    \[ \gamma_{\SL}(Q,\alpha) \geq \alpha_{\max}^{-3.5} n^{-1.5}. \]
\end{theorem}

In section 5 we show that $\gamma_{\GL}(Q,\alpha,\sigma)^{-1}$ is polynomial in $\alpha_{\max}$,$|\sigma|_{\max}$ and the number of vertices for any acyclic quiver $Q$. In particular this yields another proof of \cite[Theorem 1.1]{algorithmic_sigma}. We also show, that the dependence on $|\sigma|_{\max}$ is unavoidable.

\begin{theorem} [\Cref{mainresult_sigma_semistability} restated]
    Let $Q$ be an acyclic quiver, $\alpha$ a dimension vector and $\sigma: Q_0 \to \bZ$ a weight. The gap of the $\GL_\alpha$ action on $\Rep_\alpha(Q)\oplus \chi_\sigma$ satisfies
    \[ \gamma_{\GL}(Q,\alpha,\sigma) \geq \sqrt{\frac{4}{\alpha_{\max} n^3 (\langle |\sigma|,\alpha\rangle +2)^2}} \geq \Omega(\alpha_{\max}^{-1.5} n^{-2.5} |\sigma|_{\max}^{-1}).\]
\end{theorem}

In section 6 we show that for tree quivers the gap becomes exponentially small in the number of leaves for certain dimension vectors. In particular this shows that the results from sections 3 and 4 can not be generalized to arbitrary tree quivers.

\begin{theorem} [\Cref{tree_quivers_leaf_bound} restated]
    Let $Q$ be a tree quiver with $n$ leaves. Then there exists a dimension vector $\alpha$ with $\alpha_{\max} = n+1$, such that \[ \gamma_{\SL}(Q,\alpha) \leq \frac{1}{2^{\lfloor \frac{n+1}{2} \rfloor}-1}.\]
\end{theorem}

In section 7 we extend the previous theorem to show that for connected quivers the weight margin becomes exponentially small in the number of vertices for certain dimension vectors.

\begin{theorem}[\Cref{weight_margin_connected quivers} restated]
    Let $Q$ be a connected quiver with $n+1$ vertices. Then there exists a dimension vector $\alpha$ with $\alpha_{\max} \leq n^2$, such that
    \[ \gamma_{\ST}(Q,\alpha) \leq \frac{1}{2^{\lfloor \frac{n+1}{2}\rfloor}-1}. \]
\end{theorem}

The most interesting open question concerns the behavior of the gap for $D,\hat{D}$-type quivers, we believe it highly likely, that bounds similar to those of sections 4 exist for these quivers. Another question concerns tree quivers with many vertices but few leaves for example take a vertex with 3 strands attached, such a quiver might also have a well behaved gap.

%% file: preliminaries.tex
This section consists of two preliminary topics. First we derive an explicit formula for the moment map of the $\SL_\alpha$ action on $\Rep_{\alpha}(Q)$. Second, we introduce $\sigma$-semistability of representations of acyclic quivers and relate it to $\SL_\alpha$-semistability.

\subsection{Moment map formula}

For a quiver $Q$ and dimension vector $\alpha$, the groups $\GL_\alpha,\SL_\alpha$ can be considered as self-adjoint subgroups of $\GL_N$, where $N:= \sum_{v\in Q_0} \alpha(v)$, via a block diagonal embedding. Let $\pi:\GL_\alpha \to \GL(\Rep_\alpha(Q))$ be the natural action, then the maximal compact subgroup $K$ and its Lie algebra are
\[ K = \prod_{v\in Q_0} U_{\alpha(v)} \quad \Lie(K) = \bigoplus_{v\in Q_0} i\Herm_{\alpha(v)},\]
i.e., tuples of unitary and skew-hermitian matrices respectively. The Lie-Algebra action on $X \in \Rep_\alpha(Q)$ is
\[ \Pi(H) \cdot X = \partial_{s=0} ( e^{H_{ha} s}X_a e^{-H_{ta}s})_{a\in Q_1} = ( H_{ha} X_a - X_a H_{ta})_{a\in Q_1}.\]
For a vertex $v \in Q_0$ we denote the set of incoming, outgoing and adjacent arrows as follows
\[ v^+ := \{ a \in Q_1 \;|\; ha = v \} \quad v^- := \{ a \in Q_1 \;|\; ta = v \} \quad v^{\pm} := v^+ \cup v^-.\]

\begin{restatable}{proposition}{momentmapformula}
\label{momentmap_formula}
    For a quiver $Q$ and dimension vector $\alpha$ the moment map $\mu_{\SL_\alpha} : \Rep_\alpha(Q) \to \bigoplus_{v\in Q_0} \Herm_{\alpha(v)}$ is given by
    \begin{equation}
        \label{moment_def}
        \mu_{\SL_\alpha}(X) = \frac{1}{\| X \|^2} \left( \mu_v(X) \right)_{v\in Q_0},
    \end{equation}
    where
    \begin{equation}
        \label{moment_at_v_def}
        \mu_v(X):= \sum_{a \in v^+} p(X_aX_a^*) - \sum_{a \in v^-} p(X_a^*X_a)
    \end{equation}
    and $p:\bC^{\alpha(v)\times \alpha(v)} \to \bC^{\alpha(v)\times \alpha(v)}, X \mapsto X-\frac{\tr(X)}{\alpha(v)} I_{\alpha(v)}$ is the orthogonal projection onto the zero-trace subspace. We call $\mu_v(X)$, the \emph{moment of $X$ at $v$}. 
\end{restatable}

\begin{proof}
        For $X\in \Rep_\alpha(Q)$ and $H\in i\Lie(K)= \bigoplus_{v\in Q_0} \Herm_{\alpha(v)}$ we have
        \begin{align*}
            \frac{\langle X, \Pi(H)\cdot X\rangle}{\langle X, X \rangle} &= \frac{\sum_{a\in Q_1} \tr(X_a^* H_{ha}X_a - X_a^* X_a H_{ta})}{\|X\|^2}\\ &= \frac{\sum_{a\in Q_1} \tr(X_aX_a^* H_{ha}) - \tr(X_a^* X_a H_{ta})}{\|X\|^2} \\&= \frac{\sum_{v\in Q_0} \tr\left(\left( \sum_{a\in v^+} X_aX_a^* - \sum_{a \in v^-} X_a^*X_a \right)H_v\right)}{\|X\|^2}.
        \end{align*}
        Thus, by the defining property of the moment map we have
        \[\mu_{\GL_\alpha}(X) = \frac{1}{\|X\|^2} \left( \sum_{a\in v^+} X_aX_a^* - \sum_{a \in v^-} X_a^*X_a \right)_{v\in Q_0}\]
        We note that $\Lie(\SL_\alpha) =  \bigoplus_{v\in Q_0} \{ X\in \bC^{\alpha(v) \times \alpha(v)} \;|\; \tr(X)=0\}$. Thus, by projecting $\mu_{\GL_\alpha}(X)$ onto the trace zero matrices, we obtain the moment map of the induced $\SL_\alpha$ action and the desired result by \cite[Proposition 4.2]{FR21}.
    \end{proof}

\subsection{$\sigma$-semistability of quiver representations}

Let $Q$ be a quiver and $\alpha$ a dimension vector. A \emph{weight} is a map $\sigma:Q_0\to \bZ$. It defines a \emph{character} (i.e., a one dimensional representation) $\chi_\sigma: \GL_\alpha \to \bC^\times \simeq \GL(\bC)$ via
\[ \chi_\sigma(g) := \prod_{v\in Q_0} \det(g_v)^{\sigma(v)}.\]
A representation $V\in \Rep_\alpha(Q)$ is $\sigma$\emph{-semistable} when $(V,1)\in \Rep_\alpha(Q)\oplus \chi_\sigma$ is $\GL_\alpha$ semistable. Otherwise $V$ issues  $\sigma$\emph{-unstable}. We denote the gap of the $\GL_\alpha$ action on $\Rep_\alpha(Q)\oplus \chi_\sigma$ as $\gamma_{\GL}(Q,\alpha,\sigma)$. 
For a dimension vector $\alpha$ we write $\sigma(\alpha):=\sum_{v\in Q_0} \sigma(v)\alpha(v)$. The $\sigma$-semistability of a representation can be checked using the following Theorem, which we will use throughout the paper, for a proof we refer to \cite{ITQR}. 

\begin{theorem}[King's Criterion]
\label{KC}
    Let $Q$ be an acyclic quiver, $\alpha$ a dimension vector and $\sigma$ a weight, such that $\sigma(\alpha) =0$. Then a representation $V\in \Rep_\alpha(Q)$ is $\sigma$-semistable if and only if $\sigma(\dim(W)) \leq 0$ for all subrepresentations $W\subseteq V$.
\end{theorem}

The ring of invariants under the $\SL_\alpha$ action is called the \emph{ring of semi-invariants} and denoted as
\[ \SI(Q,\alpha) := \bC[\Rep_\alpha(Q)]^{\SL_\alpha}.\]
The \emph{ring of semi-invariants of weight} $\sigma$ is defined as
\[ \SI(Q,\alpha,\sigma) := \{ f\in \bC[\Rep_\alpha(Q)] \;|\; g\cdot f = \chi_\sigma(g) f \text{ for all } g\in \GL_\alpha \}.\]
It is clear, that $\SI(Q,\alpha,\sigma) \subseteq \SI(Q,\alpha)$. Moreover, the ring of semi-invariants splits as a direct sum of the rings of semi-invariants of weight $\sigma$, where $\sigma$ ranges over all weights, see for example \cite[Lemma 10.4.1]{ITQR}.
\begin{equation}
\label{semiinvariant_ring_decomposition}
    \SI(Q,\alpha) = \bigoplus_{\sigma : Q_0\to \bZ} \SI(Q,\alpha, \sigma).
\end{equation}
The semistability and $\sigma$-semistability of $V\in \Rep_\alpha(Q)$ are closely related by the following lemma.

\begin{lemma}
\label{sigma_unstable_relation}
    Let $V\in \Rep_\alpha(Q)$, then $V$ is $\SL_\alpha$-unstable if and only if it is $\sigma$-unstable for each non-zero weight $\sigma \in \bZ^{Q_0}\backslash\{0\}$. \JM{i would like to cite this result, but i could not find a source, it is not in \cite{ITQR} i think}
\begin{proof}
    For any weight $\sigma$ we have an isomorphism of $\GL_\alpha$ modules
    \[ \bC[\Rep_\alpha(Q) \oplus \chi_\sigma] \simeq  \bC[\Rep_\alpha(Q)] \otimes \bC[\chi_\sigma] \simeq \bC[x] \otimes \bC[y].\]
    Thus any polynomial $h(x,y)$ in this ring may be expressed as
    \[ h(x,y) = \sum_{k\in \bN} f_k(x) y^k\]
    with $f_k(x) \in \bC[\Rep_\alpha(Q)]$ and all but finitely many summands vanishing. The action of $\GL_\alpha$ on this polynomial is given by
    \[ g\cdot h(x,y) =  \sum_{k\in \bN} (g\cdot f_k)(x) (\chi_\sigma(g^{-1}))^k y^k = \sum_{k\in \bN} (g\cdot f_k)(x) (\chi_{k\sigma}(g))^{-1} y^k.\]
    Thus $h(x,y)$ is an invariant if and only if $g\cdot f_k = \chi_{k\sigma}(g) f$ for every $k\in \bN$ and $g\in \GL_\alpha$, i.e., when every $f_k$ is a semi invariant of weight $k\sigma$. This shows that the invariant ring splits as a direct sum
    \[ \bC[\Rep_\alpha(Q) \oplus \chi_\sigma]^{\GL_\alpha} \simeq \bigoplus_{k\in \bN} \SI(Q, \alpha, k\sigma).\]
    Therefore $(V,1) \in \Rep_\alpha(Q) \oplus \chi_\sigma$ is $\GL_\alpha$ unstable if and only if each semi-invariant of weight $k\sigma$ vanishes on $V$ for each $k> 0$. Similarly since we also have
    \[ \SI(Q,\alpha) = \bigoplus_{\sigma : Q_0\to \bZ} \SI(Q,\alpha, \sigma), \tag{\ref{semiinvariant_ring_decomposition}}\]
    we know that $V \in \Rep_\alpha(Q)$ is $\SL_\alpha$ unstable if and only if every semi-invariant of weight $\sigma$ vanishes on $V$ for each $\sigma\neq 0$. Combining these two facts yields that $V$ is $\SL_\alpha$ unstable if and only if $(V,1) \in \Rep_\alpha(Q) \oplus \chi_\sigma$ is $\GL_\alpha$ unstable for each non-zero $\sigma$.
\end{proof}
\end{lemma}

%% file: tree_quivers.tex
A \emph{tree quiver} is one whose underlying undirected graph is a tree. In this section we prove the following lower bound for the gap of tree quivers with uniform dimension vectors. Remarkably the bound does not depend on the orientation of $Q$.

\begin{restatable}{theorem}{mainresulttreequivers}
\label{mainresult_treequivers}
    Let $Q$ be a tree quiver with $n+1$ vertices and uniform dimension vector $\alpha = (d,\ldots, d)$. Then the gap of the $\SL_{\alpha}$ representation $\Rep_{\alpha}(Q)$ is lower bounded by 
    \[ \gamma_{\SL}(Q, \alpha) \geq \sqrt{\frac{6}{d(d-1) n(n+1)(2n+1)}} \geq \Omega(n^{-1.5}d^{-1}). \]
\end{restatable} 
From now on we fix a tree quiver $Q$, a uniform dimension vector $\alpha = (d,\ldots, d)$ and an unstable representation $X\in \Rep_\alpha(Q)$. To this representation we associate the function $y: Q_1 \to \bR_{\geq 0}$ defined by
\[ y_a := \|p(X_a^*X_a)\| = \|p(X_aX_a^*)\|, \quad a \in Q_1.\]
Let us fix a leaf $\sigma \in Q_0$ and consider the tree quiver $\Tilde{Q}$ obtained by orienting each arrow in $Q$, such that they point towards $\sigma$. There is a natural map $o : \Tilde{Q}_0\backslash\{\sigma\} \to \Tilde{Q}_1$, mapping any $v\neq \sigma$ to its unique outgoing arrow. We define the function $f_{\Tilde{Q}} : \bR^{\Tilde{Q}_1}\backslash\{0\} \to \bR$ via
\begin{equation}
    \label{helper_tree_quivers_def}
    f_{\Tilde{Q}}(y) :=\frac{\sum_{v \in \Tilde{Q}_0 \backslash \{ \sigma \}} ( y_{o(v)} - \sum_{a\in v^+} y_a)_+^2}{(\sum_{a\in \Tilde{Q}_1} y_a)^2},
\end{equation}
where $x_+ = \max(x,0)$ denotes the positive part of a real number $x$. The proof of the main theorem now consists of three steps.
\begin{enumerate}
    \item First we lower bound the norm of the moment map in terms of $f_{\Tilde{Q}}(y)$
    \[ \|\mu_{\SL}(X)\|^2 \geq \frac{f_{\Tilde{Q}}(y)}{d(d-1)}.\]
    \item Next we inductively reduce to the case when $\Tilde{Q}$ is the equioriented $A_n$ quiver.
    \item Finally we use analytic methods to lower bound $f_{A_n}(y)$
    \[ f_{A_{n+1}}(y) \geq \frac{6}{n(n+1)(2n+1)}.\]
    
\end{enumerate}

\begin{example}
    Let $\Tilde{Q}$ be the following rooted tree quiver with the sink $\sigma$ marked in red.
    \[\begin{tikzcd}
	& \textcolor{red}{\bullet} \\
	& \bullet \\
	\bullet && \bullet & \bullet \\
	&& \bullet
	\arrow["4", from=3-1, to=2-2]
	\arrow["5", from=2-2, to=1-2]
	\arrow["3", from=3-3, to=2-2]
	\arrow["1", from=4-3, to=3-3]
	\arrow["2", from=3-4, to=3-3]
\end{tikzcd}\]
\[ f_{\Tilde{Q}}(x) = \frac{(x_5-x_4-x_3)_+^2 + x_4^2 + (x_3-x_2-x_1)_+^2 + x_2^2+x_1^2}{(x_1+x_2+x_3+x_4+x_5)^2} \]
\end{example}

\subsection*{Step 1}

We begin by describing the nullcone of tree quivers with uniform dimension vectors.

\begin{proposition}
\label{nullcone_tree_uniform}
    Let $Q$ be a tree quiver and $\alpha = (d, \ldots, d)$ a uniform dimension vector, then
    \[ \NC(Q, \alpha) = \{ X\in \Rep_\alpha(Q) \;|\; \det(X_a) = 0, \;\forall a \in Q_1 \} \]
    \begin{proof}
        Consider the polynomials $\det_a \in \bC[\Rep_\alpha(Q)]$ for $a\in Q_1$ defined by $\det_a(X) = \det(X_a)$. Clearly each of these is an $\SL_\alpha$ invariant which shows the inclusion from left to right. Conversely let $X \in \Rep_\alpha(Q)$ is a representation such that $\det(X_a)= 0$ for all $a \in Q_1$. We choose some leaf $u \in Q_0$ with outgoing arrow $a$ (the case of an incoming arrow works similarly). By acting at $u$ with a unitary transformation we may assume that $e_1 \in \ker(X_a)$. Now we act on $X$ with the one parameter subgroup $M : \bC^\times \to \GL_{\alpha}$, defined by $M_w(t) = I_d$ for $w \neq u$ and $M_v(t) = \diag(t^{1-d}, t, \ldots, t)$ to obtain $Y:= \lim_{t\to \infty} M(t) \cdot X$ with $Y_a = 0$ and $Y_b = X_b$ for all $b\in Q_0\backslash\{a\}$. This way we can inductively kill all arrows, which shows that $X$ is unstable.
    \end{proof}
\end{proposition}

Recall the formula for the moment map from \Cref{moment_def}
\[ \|\mu_{\SL}(X)\|^2 = \frac{\sum_{v\in Q_0} \|\mu_v(X)\|^2}{\|X\|^4}.\]
We aim to bound the norm of the moment map in terms of $y$. For the denominator we use the following upper bound.
\begin{lemma}
\label{norm_X_bound_in_terms_of_y}
    We have
    \[ \|X\|^2 \leq \sqrt{d(d-1)} \sum_{a\in Q_1} y_a. \]
\begin{proof}
    Let $a \in Q_1$ and denote $A := X_a^*X_a$ as well as the eigenvalues of $A$ as $\lambda_1 \geq \ldots \geq \lambda_d$. Since $X$ is unstable we know by \Cref{nullcone_tree_uniform}, that $A$ has vanishing determinant, so $\lambda_d = 0$. Using the Cauchy-Schwarz inequality, we obtain
    \[ \tr(A)^2 = \Bigg(\sum_{i=1}^{d-1} \lambda_i\Bigg)^2 \leq (d-1) \sum_{i=1}^{d-1} \lambda_i^2 = (d-1)\|A\|^2.\]
    Using the inequality above in the third step, we see
    \[ \|p(A)\|^2 = \sum_{i=1}^{d} \Big(\lambda_i - \frac{\tr(A)}{d}\Big)^2 = \|A\|^2 - \frac{\tr(A)^2}{d} \geq \frac{\tr(A)^2}{d-1} -\frac{\tr(A)^2}{d} = \frac{\tr(A)^2}{d(d-1)}.\]
    This yields
    \[ y_a = \|p(X_a^*X_a)\| \geq \frac{\tr(X_a^*X_a)}{\sqrt{d(d-1)}} = \frac{\|X_a\|^2}{\sqrt{d(d-1)}}. \]
    The result follows by summing over $a\in Q_1$.
\end{proof}
\end{lemma}
For the numerator of the norm of the moment map, we use the following lemma. Recall the definition of the moment at a vertex $v \in Q_0$ from \Cref{moment_at_v_def}.
\[ \mu_v(X) = \sum_{a \in v^+} p(X_aX_a^*) - \sum_{a \in v^-} p(X_a^*X_a).\]

\begin{lemma}
\label{norm_moment_at_v_bound_in_terms_of_y}
    Let $v\in Q_0$ and $o\in Q_1$ be some arrow adjacent to $v$. Then
    \[ \|\mu_v(X)\| \geq \Bigg( y_o - \sum_{\substack{a\in v^\pm \\ a\neq o} } y_a \Bigg)_+, \]
    where $x_+ = max(x,0)$ denotes the positive part of a real number. 
\begin{proof}
    W.l.o.g. assume that $o$ is an outgoing arrow, i.e. $to = v$. By the triangle inequality we have
    \[ \Bigg\| \sum_{a\in v^+} p(X_aX_a^*) - \sum_{\substack{a\in v^- \\ a\neq o}} p(X_a^*X_a) \Bigg\| \leq \sum_{\substack{a\in v^\pm \\ a\neq o}} \|p(X_aX_a^*)\| = \sum_{\substack{a\in v^\pm \\ a\neq o}} y_a. \tag{$\ast$} \]
    By using the reverse triangle inequality in the second step and $(\ast)$ in the fourth step we obtain the inequality
    \begin{align*}
        \|\mu_v(X)\| &= \|\sum_{a \in v^+} p(X_aX_a^*) - \sum_{a \in v^-} p(X_a^*X_a)\|\\ &\geq  \Bigg|\|p(X_o^*X_o)\| - \Bigg\|\sum_{a\in v^+ } p(X_aX_a^*) - \sum_{\substack{a\in v^-\\a \neq o} } p(X_a^*X_a)\Bigg\| \Bigg| \\ &\geq \Bigg(\|p(X_o^*X_o)\| - \Bigg\|\sum_{a\in v^+} p(X_a X_a^*) - \sum_{\substack{a\in v^-\\a \neq o}} p(X_a^*X_a)\Bigg\| \Bigg)_+ \\ &\geq \Bigg(y_o - \sum_{\substack{a\in v^\pm \\ a\neq o}} y_a\Bigg)_+ \qedhere
    \end{align*}
\end{proof}
\end{lemma}

Now we may combine the bounds from \Cref{norm_X_bound_in_terms_of_y} and \Cref{norm_moment_at_v_bound_in_terms_of_y} to obtain the following lemma

\begin{lemma}
\label{norm_moment_bound_f_Q}
    Let $\sigma \in Q_0$ be a leaf and let $\Tilde{Q}$ be the quiver obtained by orienting all arrows in $Q$, such that they point towards $\sigma$. Then we have
    \[ \|\mu_{\SL}(X)\|^2 \geq \frac{f_{\Tilde{Q}}(y)}{d(d-1)}.\]
\begin{proof}
    Recall the formula for the moment map from \Cref{moment_def}
    \[ \|\mu_{\SL}(X)\|^2 = \frac{\sum_{v\in Q_0} \|\mu_v(X)\|^2}{\|X\|^4}.\]
    The result follows by applying \Cref{norm_X_bound_in_terms_of_y} for the denominator and applying \Cref{norm_moment_at_v_bound_in_terms_of_y} for each summand (with $v\neq \sigma$) in the numerator, where the outgoing arrow $o$ is the unique outgoing arrow in $\Tilde{Q}$.
\end{proof}
\end{lemma}

\subsection*{Step 2}

By the previous lemma it suffices to find a lower bound for $f_{\Tilde{Q}}(y)$ for any tree quiver $\Tilde{Q}$ which is rooted at a leaf $\sigma$. We accomplish this goal by reducing to the case of $\Tilde{Q}$ being the equioriented $A_{n+1}$ quiver via so called sliding moves. When $a$ and $b$ are arrows pointing to the same vertex $w$ we define a new rooted tree quiver $S_a^b(\Tilde{Q})$ by changing the head of $a$ to be the tail of $b$ as in the diagram below
\[\begin{tikzcd}
	& w \\
	u \arrow["a", ur] && v \arrow["b"', ul]
\end{tikzcd} \quad \longrightarrow \quad \begin{tikzcd}
	& w \\
	u \arrow["a", rr] && v \arrow["b"', ul]
\end{tikzcd}\]
    We also define an associated linear map
    \[ S_a^b : \bR^{\Tilde{Q}_1} \to \bR^{\Tilde{Q}_1},\quad (y_c)_{c\in \Tilde{Q}_1} \mapsto (z_c)_{c\in \Tilde{Q}_1} \quad \text{with } z_c := \left\{ \begin{array}{cc}
        y_b+y_a & ,\text{ if }c=b \\
        y_c & ,\text{ otherwise}
    \end{array} \right. \]
    This map is illustrated in the picture below.
    \[\begin{tikzcd}
	& w \\
	u \arrow["y_a", ur] && v \arrow["y_b"', ul]
\end{tikzcd} \quad \longrightarrow \quad \begin{tikzcd}
	& w \\
	u \arrow["y_a", rr] && v \arrow["y_a+y_b"', ul]
\end{tikzcd}\]

\begin{lemma}
    \label{sliding_decreases_f}
    For any $y \in \bR^{\Tilde{Q}_1}_{\geq 0}$, we have
    \[ f_{\Tilde{Q}}(y) \geq f_{S_a^b(\Tilde{Q})}(S_a^b(y)).\]
\begin{proof}
    Denote $z := S_a^b(y)$. The numerator of both sides is the same since 
\[ y_{o(v)} - \sum_{c\in v^+} y_c = y_b - R = (y_a + y_b) - y_a - R = z_{b} - z_a - R = z_{o(v)} - \sum_{c\in v^+} z_c,\]
where $R := \sum_{c\in v^+\backslash\{a\}} y_c = \sum_{c\in v^+\backslash\{a\}} z_c$ and similar equalities hold for the other vertices. The denominator of the left hand side is less than or equal to the denominator of the right hand side since
\[ \sum_{c\in \Tilde{Q}_1} y_c \leq y_a + y_b + \sum_{\substack{c\in \Tilde{Q}_1 \\ c \neq b}} y_c = z_b + \sum_{\substack{c\in \Tilde{Q}_1 \\ c \neq b}} z_c = \sum_{c\in \Tilde{Q}_1} z_c.\]
Combining these two observations we obtain
\[ f_{\Tilde{Q}}(y) = \frac{\sum_{v \in \Tilde{Q}_0 \backslash \{ \sigma \}} ( y_{o(v)} - \sum_{a\in v^+} y_c)_+^2}{(\sum_{a\in \Tilde{Q}_1} y_c)^2} \geq \frac{\sum_{v \in \Tilde{Q}_0 \backslash \{ \sigma \}} ( z_{o(v)} - \sum_{a\in v^+} z_c)_+^2}{(\sum_{a\in \Tilde{Q}_1} z_c)^2} = f_{S_a^b(\Tilde{Q})}(S_a^b(y)) \qedhere \]
\end{proof}
\end{lemma}

\begin{lemma}
\label{f_Q_geq_f_A}
    Let $\Tilde{Q}$ be a tree quiver with $n+1$ vertices, which is rooted at a leaf, then
    \[ \min_{y\in \bR^{\Tilde{Q}_1}_{\geq 0}} f_{\Tilde{Q}}(y) \geq \min_{y \in \bR^{n+1}_{\geq 0}} f_{A_{n+1}}(y).\]
\begin{proof}
    Let $y\in \bR^{\Tilde{Q}_1}_{\geq 0}$. In case there exist arrows $a,b \in \Tilde{Q}_1$ with the same head $w$, we produce a new quiver $\Tilde{Q}^1 := S_a^b(\Tilde{Q})$ and a new vector $y^1 := S_a^b(y)$. We repeat this process to produce a sequence of quivers $\Tilde{Q}, \Tilde{Q}^1, \ldots, \Tilde{Q}^N$ and a sequence of vectors $y, y^1,\ldots, y^n \in \bR^{\Tilde{Q}_1}_{\geq 0}$, such that $\Tilde{Q}^{k+1}$ is obtained from $\Tilde{Q}^k$ by a sliding move. This process terminates since the sum of the distances of each vertex to the root $\sigma$ is strictly increasing under sliding moves. By \Cref{sliding_decreases_f} we have a sequence of inequalities
    \[ f_{\Tilde{Q}}(y) \geq f_{\Tilde{Q}^1}(y^1) \geq \ldots \geq f_{\Tilde{Q}^N}(y^N).\]
    Since no two arrows have the same head in $\Tilde{Q}^N$, any vertex in $\Tilde{Q}^N$ has in-degree at most $1$ and since it is a rooted tree quiver any vertex has out-degree at most $1$ as well. Thus $\Tilde{Q}^N$ must be the equioriented $A_{n+1}$ quiver and we are done.
\end{proof}
\end{lemma}

\subsection*{Step 3}

\begin{lemma}
\label{f_A_bound}
    Let $n\in \bN$ and let $\Omega = \{y \in \bR^{n}_{\geq 0} \;|\; \sum y_k \neq 0\}$, then we have
    \[ \min_{y \in \Omega} f_{A_{n+1}}(y) = \frac{6}{n(n+1)(2n+1)}.\]
\begin{proof}
    We label the vertices and arrows of the equioriented $A_{n+1}$ quiver as follows.
    \[\begin{tikzcd}
	1 & 2 & 3 & \ldots & {n+1}
	\arrow["1", from=1-1, to=1-2]
	\arrow["2", from=1-2, to=1-3]
	\arrow["3", from=1-3, to=1-4]
	\arrow["n", from=1-4, to=1-5]
    \end{tikzcd}\]
    Expressing $f(y) := f_{A_{n+1}}(y)$ with these labels yields
    \[ f(y) = \frac{y_1^2+ \sum_{k=2}^n (y_k - y_{k-1})_+^2}{(\sum_{k=1}^n y_k)^2}\]
    Consider the variables $x_1 := y_1$ and $x_k:= (y_k-y_{k-1})_+$ for $k\geq 2$. For any $k$ we have $y_k \leq \sum_{l = 1}^k x_k$ and thus
    \[ \sum_{k=1}^n y_k \leq \sum_{k=1}^n (n+1-k) x_k.\]
    It follows that
    \[ f(y) \geq \frac{\sum_{k=1}^n x_k^2}{(\sum_{k=1}^n (n+1-k) x_k)^2} \geq \frac{\sum_{k=1}^n x_k^2}{(\sum_{k=1}^n (n+1-k)^2)(\sum_{k=1}^n x_k^2)} = \frac{1}{\sum_{k=1}^n k^2} = \frac{6}{n(n+1)(2n+1)},\]
    where we used the Cauchy-schwarz inequality in the second step.
\end{proof}
\end{lemma}

Finally we prove the main result by combining the last 3 lemmata.
\mainresulttreequivers*
\begin{proof}

Let $Q$ be a tree quiver and $\alpha= (d,\ldots, d)$ a uniform dimension vector. Let $X\in \Rep_\alpha(Q)$ be unstable and denote $y_a := \|p(X_a^*X_a)\|$ for $a\in Q_1$. Let $\sigma \in Q_0$ be a leaf and let $\Tilde{Q}$ be the quiver obtained by orienting all arrows in $Q$, such that they point towards $\sigma$. Then we have
\[ \|\mu_{\SL}(X)\|^2 \geq \frac{f_{\Tilde{Q}}(y)}{d(d-1)} \geq \frac{6}{d(d-1)n(n+1)(2n+1)},\]
where we used \Cref{norm_moment_bound_f_Q} in the first step and \Cref{f_Q_geq_f_A} and \Cref{f_A_bound} in the second step. Since $X$ is an arbitrary unstable representation we obtain the result
\[ \gamma_{\SL}(Q, \alpha) \geq \sqrt{\frac{6}{d(d-1) n(n+1)(2n+1)}} \geq \Omega(n^{-1.5}d^{-1}). \qedhere\]
\end{proof}

%% file: A_and_extended_A.tex
The goal of this section is to prove the following theorem, which yields lower bounds on the gap of quivers of type $A$ and $\hat{A}$.

\begin{restatable}{theorem}{Mainresult}
\label{mainresult}

Let $Q$ be a quiver of type $A_{n}$ or $\hat{A}_n$ and let $\alpha$ be a dimension vector.
\begin{enumerate}
    \item In case $Q$ is of type $A$ we have
\[ \gamma_{\SL}(Q,\alpha) \geq \frac{1}{\sqrt{2}} \alpha_{\max}^{-2} n^{-1.5}. \]
    \item In case $Q$ is of type $\hat{A}$ and acyclic we have
\[ \gamma_{\SL}(Q,\alpha) \geq  \alpha_{\max}^{-3.5} n^{-1.5}. \]
    \item In case $Q$ is the oriented cycle, we have
\[ \gamma_{\SL}(Q,\alpha) \geq \alpha_{\max}^{-3} n^{-1.5}. \]
\end{enumerate}
\end{restatable}

In section 4.1 we introduce a notion of angular distance between subspaces $U,W$ of a Euclidean space, which we denote by $|\langle U, W\rangle|^2$ and develop some techniques for working with this notion. In section 4.2 we introduce an important lower bound \Cref{lem:difference_of_hermitians_bound} which will be used to bound the norm of the moment of $X\in \Rep_\alpha(Q)$ at a vertex $v$ in terms of the spectra of the $X_i^*X_i$ as well as the angular distance of their eigenspaces. When summing up over $\|\mu_v(X)\|^2$ and applying the bound \Cref{lem:difference_of_hermitians_bound} we need to keep track of the summands appearing on the right hand side. In order to do that, we construct a graph $G$ and an associated function $f_G$, such that
\[ 4\|\mu(X)\|^2 \geq \alpha_{\max}^{-2} f_{G}(s,\lambda),\]
for certain values of $s$ and $\lambda$. In \Cref{lem:graph_component_inequality} we proceed to bound the right hand side in terms of $f_{H_i}(s,\lambda)$, where $H_i$ range over the connected components of $G$. Some connected components yield weights $\sigma$ for which $\sigma(\alpha) =0$, then Kings-criterion \ref{KC} and \Cref{sigma_unstable_relation} yield a subrepresentation $Y\subseteq X$, such that $\sigma(\dim(Y)) \geq 1$. This is the crucial tool to let us bound $f_{H_i}(s,\lambda)$ in \Cref{lem:graph_component_inequality}.

Throughout this section $Q$ will be a quiver of type $\hat{A}_n$ and we begin by defining our labeling of the vertices and arrows.

\begin{definition}
    For any quiver $Q$ of type $\hat{A}_n$ we label the vertices and arrows as follows
    \[ Q_0:=\bZ/n\bZ\quad Q_1:=\{1,\ldots, n\}.\]
    We define the head and tail maps as follows
    \[ (ta,ha):=\left\{ \begin{array}{cl}
        (a,a+1) & \text{, if $a$ points right} \\
        (a+1,a) & \text{, if $a$ points left}
    \end{array}\right. \]
\end{definition}

Notice that both vertices and arrows are labeled by integers, this will make notation later on much easier. For example it allows us to define the head and tail maps through simple arithmetic.

\subsection{Angles and subspaces}

We begin by defining angles between vectors and subspaces and introducing a notion of angular distance between subspaces. Let $V\subseteq \bC^{d}$ be a subspace and let $u,v\in \bC^d\backslash\{0\}$ be vectors. The \emph{angle of $u$ and $v$} is defined as
\[ \angle(u,v) := \cos^{-1} \Big( \frac{|\langle u, v \rangle| }{\|u\| \cdot \|v\|} \Big), \]
with the inverse being chosen in the interval $[0, \pi[$. We define the angles $\angle(v,0)=\angle(0,v):= \frac{\pi}{2}$. The \emph{angle of $V$ and $v$} is defined as 
\[\angle(V,v) := \angle(\pi_V(v), v),\]
where $\pi_V : \bC^d \to \bC^d$ is the orthogonal projection onto $V$. Notice that $\angle(V,v) \in [0,\frac{\pi}{2}]$. We made the unusual definition $\angle(0,v):= \frac{\pi}{2}$, since we want to ensure that $v\in V^{\perp} \Leftrightarrow \angle(V,v)=\frac{\pi}{2}$. Now let $U,W\subseteq \bC^d$ be two subspaces and let $u_1,\ldots, u_m$ and $w_1,\ldots, w_n$ be orthonormal bases for $U,W$ respectively. We define
\[ |\langle U, W \rangle|^2 := \sum_{i=1}^m \sum_{j=1}^n |\langle u_i,w_j\rangle|^2. \]
We note, that this definition doesn't depend on the choice of unitary bases, which follows from \Cref{lem:subspace_similarity_equations} below. We also notice, that
\[ |\langle U, W \rangle|^2 = 0 \quad \iff \quad U \perp W.\]
Thus $|\langle U, W \rangle|^2$ can be thought of as a measure for how close the two subspaces are to being orthogonal. In this section we develop a variety of lemmata relating the various notions defined above. 





\begin{lemma}
\label{lem:subspace_similarity_equations}
    Let $U,V\subseteq \bC^d$ be subspaces and $u_1,\ldots, u_m \in U$ and $v_1,\ldots, v_n\in V$ orthonormal bases, then the following equalities hold
    \[ |\langle U,V\rangle|^2 = \tr(\pi_U \pi_V) = \sum_{i=1}^m \norm{\pi_V(u_i)}^2 = \sum_{i=1}^m \cos^2(\angle (V,u_i)) = \sum_{j=1}^n \norm{\pi_U(v_j)}^2 =  \sum_{j=1}^n \cos^2(\angle(U,v_j)).\]
    Further there exist orthonormal bases $u_1,\ldots, u_m \in U$ and $v_1,\ldots, v_n\in V$ such that $\langle u_i,v_j\rangle = 0$ whenever $i\neq j$. The angles $\angle(u_i,v_i)$ are called the principal angles of $U$ and $V$ \cite[Theorem 2.1]{Zhu_2013}.

    \begin{proof}
    For the first equality, we calculate
    \begin{align*}
        \tr(\pi_U\pi_V) &= \tr\left( \left(\sum_{i=1}^m u_iu_i^* \right)\left(\sum_{j=1}^n v_jv_j^* \right)\right)\\ &= \tr\left( \sum_{i=1}^m \sum_{j=1}^n u_iu_i^*v_jv_j^* \right)\\ &= \sum_{j=1}^n \sum_{i=1}^m \tr(u_iv_j^*u_i^*v_j)\\ &= \sum_{j=1}^n \sum_{i=1}^m |\langle u_i, v_j\rangle|^2. 
    \end{align*}
    
    For the second equality we notice
    \[ \langle u_i, \pi_V(u_i)\rangle = \langle \pi_V(u_i) + \pi_{V^\perp}(u_i), \pi_V(u_i)\rangle = \langle \pi_V(u_i), \pi_V(u_i)\rangle = \norm{\pi_V(u_i)}^2\]
    and we obtain
    \[ \tr(\pi_U \pi_V) = \sum_{i=1}^m \tr(u_i u_i^* \pi_V) = \sum_{i=1}^m u_i^* \pi_V u_i = \sum_{i=1}^m \langle u_i, \pi_V(u_i)\rangle = \sum_{i=1}^m \norm{\pi_V(u_i)}^2.\]
    
    For the third equality, we calculate
    \[\cos^2(\angle(V,u_i)) = \frac{\langle u_i, \pi_V(u_i)\rangle^2}{\norm{\pi_V(u_i)}^2} = \norm{\pi_V(u_i)}^2.\]
    The fourth and fifth equality follow by exchanging $U$ and $V$.
    
    For the last statement consider the singular value decomposition of $\pi_U\circ \pi_V : V\to U$. There exist orthonormal bases $u_1,\ldots, u_m \in U$ and $v_1,\ldots, v_n\in V$ such that $\pi_U\pi_V v_i = \lambda_i u_i$ for some $\lambda_i\in \bR$. We obtain $\lambda_j u_j = \pi_U v_j = \sum_{i=1}^m u_i u_i^* v_j = \sum_{i=1}^m \langle u_i, v_j\rangle u_i $ and thus $\langle u_j,v_i\rangle = 0$ for $i\neq j$.
    \end{proof}
    
\end{lemma}

Some more basic properties include the following

\begin{lemma}
    Let $U,V,W\subseteq \bC^d$ be subspaces. Then the following statements hold.
    \begin{enumerate}
        \item If $U\subseteq V$, we have $|\langle U,V\rangle|^2 = \dim(U)$.
        \item If $U\perp W$, we have $|\langle U+W,V\rangle|^2 = |\langle U,V\rangle|^2 + |\langle W,V\rangle|^2$
        \item We have $|\langle U,V\rangle|^2 = \dim(V) - |\langle U^\perp,V\rangle|^2$.
    \end{enumerate}
    \begin{proof}
    \begin{enumerate}
        \item We have $\pi_U\pi_V = \pi_U$ and thus $|\langle U,V\rangle|^2 = \tr(\pi_U\pi_V)= \tr(\pi_U) =\dim(U)$.
        \item We have $\pi_{U+W} = \pi_U+\pi_W$ and thus \[|\langle U+W,V\rangle|^2 = \tr(\pi_V\pi_{U+W}) = \tr(\pi_V\pi_U) +\tr(\pi_V\pi_W) =|\langle U,V\rangle|^2 + |\langle W,V\rangle|^2. \]
        \item By 1. and 2., we have $|\langle U,V\rangle|^2 + |\langle U^\perp,V\rangle|^2 = |\langle \bC^d,V\rangle|^2 =  \dim(V)$. \qedhere
    \end{enumerate}
    \end{proof}
\end{lemma}

The following inequality is a well known result from elementary geometry. For nonzero vectors it is the triangle inequality for the Riemannian metric on spheres and if any vector is zero, it can be checked easily.
\begin{proposition}[Triangle inequality for angles]
\label{prop:triangle_for_angles}
    Let $u,v,w\in \bC^d$. Then
    \[ \angle(u,w) \leq \angle(u,v) + \angle(v,w).\]
\end{proposition}

\begin{lemma}
\label{lem:angle_inequality_3_vectors}
    Let $u,v,w\in \bC^d$ have norm $1$, such that $|\langle u,v\rangle|^2 + |\langle v,w\rangle|^2 \geq 1$, then we have
    \[ |\langle u,w\rangle|^2 \geq (|\langle u,v\rangle|^2 + |\langle v,w\rangle|^2 -1)^2. \]
    \begin{proof}
        Let $\alpha:= \angle(u,v)$ and $\beta:=\angle(v,w)$ and $\gamma:= \angle(u,w)$. Then we have
        \[ |\langle u,v\rangle|^2 = \cos^2(\alpha) \quad |\langle v,w\rangle|^2 = \cos^2(\beta) \quad |\langle u,w\rangle|^2 = \cos^2(\gamma).\]
        The condition $|\langle u,v\rangle|^2 + |\langle v,w\rangle|^2 \geq 1$ implies $\cos^2(\alpha) \geq \sin^2(\beta)$ which yields $\cos(\alpha) \geq \sin(\beta) = \cos(\frac{\pi}{2} - \beta)$, but since $\cos$ is decreasing on the interval $[0,\frac{\pi}{2}]$, we obtain $\alpha \leq \frac{\pi}{2}-\beta$, which implies $\alpha+\beta \leq \frac{\pi}{2}$. By the triangle inequality for angles we obtain $\gamma \leq \alpha+\beta$ and since $\cos$ is decreasing on the interval $[0,\frac{\pi}{2}]$ we obtain
        \[ \cos^2(\gamma) \geq \cos^2(\alpha+\beta) \geq (\cos^2(\alpha) + \cos^2(\beta)-1)^2,\]
        where the last inequality follows from the trigonometric fact, that
        \[ \frac{(\cos^2(\alpha) + \cos^2(\beta)-1)^2}{\cos^2(\alpha+\beta)} = \cos^2(\alpha-\beta) \leq 1. \qedhere\]
    \end{proof}
\end{lemma}

We can generalize the previous lemma to the following statement on arrangements of 3 subspaces.

\JM{ i removed the condition or $|\langle U,V\rangle|^2 + |\langle V,W\rangle|^2 \geq \dim(V)$ from the following lemma as i don't think its needed}
\begin{lemma}
\label{lem:3_subspace_inequality}
    Let $U,V,W\subseteq \bC^d$. If $\dim(U)+\dim(W)=d$ we have
    \[ |\langle U,W \rangle|^2 \geq \frac{1}{\dim(V)} (|\langle U,V\rangle|^2 + |\langle V,W\rangle|^2 - \dim(V))^2. \]

    \begin{proof}
    \JM{properly track indices here}
        W.l.o.g. assume $x:= |\langle U,V\rangle|^2 + |\langle V,W\rangle|^2 - \dim(V) \geq 0$. This is possible since we can pass to the orthogonal complements of $U$ and $W$ because
        \[ |\langle U^\perp,W^\perp \rangle|^2 = \dim(W^\perp) - |\langle U, W^\perp \rangle|^2 = \dim(W^\perp) - \dim(U) +  |\langle U, W \rangle|^2 = |\langle U, W \rangle|^2.\]
        and similarly
        \begin{align*}
            |\langle U^\perp,V\rangle|^2 + |\langle V,W^\perp\rangle|^2 - \dim(V) &= \dim(V) - |\langle U,V\rangle|^2 + \dim(V) - |\langle V,W\rangle|^2 - \dim(V)\\ &= -(|\langle U,V\rangle|^2 + |\langle V,W\rangle|^2 - \dim(V)).
        \end{align*}
        W.l.o.g. assume $\dim(U) \geq \dim(W)$. Let $u_1,\ldots, u_{\dim(U)}\in U$ and $v_1,\ldots, v_{\dim(V)}\in V$ be orthonormal bases, realizing the principal angles i.e. $\langle u_i,v_j\rangle = 0$ whenever $i\neq j$. And let $w_i:=\frac{\pi_W(v_i)}{\|\pi_W(v_i)\|}$ in case $v_i\not\in W^\perp$ or $w_i:=0$ in case $v_i\in W^\perp$. By \Cref{lem:subspace_similarity_equations} we obtain
        \begin{align*}
            |\langle U,V \rangle|^2 &= \sum_{i} |\langle u_i,v_i\rangle|^2, \\
            |\langle V,W \rangle|^2 &= \sum_{i=1}^{\dim(V)} \cos^2(\angle (v_i,W)) = \sum_{i=1}^{\dim(V)} |\langle v_i,w_i \rangle|^2, \\
            |\langle U,W \rangle|^2 &= \sum_i \cos^2(\angle(u_i,W)) \geq \sum_i \cos^2(\angle(u_i,w_i)) = \sum_{i} |\langle u_i, w_i\rangle|^2.
        \end{align*}
        For $i\leq \dim(U)$ let $x_i:= |\langle u_i,v_i \rangle|^2 + |\langle v_i,w_i \rangle|^2-1$ such that $x=\sum_{i} x_i$. In particular $\sum_{i: x_i>0} x_i \geq x$ and we obtain
        \[ \sum_{i: x_i>0} x_i^2 \geq \frac{x^2}{\dim(V)}.\]
        Finally we see
        \[ |\langle U,W \rangle|^2 \geq \sum_{i=1}^{\dim(U)} |\langle u_i,w_i\rangle|^2 \geq \sum_{i: x_i>0} x_i^2  \geq \frac{x^2}{\dim(V)} \]
        where we used \Cref{lem:angle_inequality_3_vectors} in the second step.
    \end{proof}
\end{lemma}

Finally the following lemma tells us how similarity of subspaces changes under linear maps.

\begin{lemma}
\label{lem:HeadTail_inequality_technique}
    Let $X\in \bC^{d'\times d}$ and let $U_k, U_k'$ be the span of the first $k$ eigenvectors of $X^*X,XX^*$ respectively and assume that the corresponding eigenvalues are nonzero i.e. $U_k\cap \ker(X)=\{0\}$. Let $V\subseteq \bC^d$ be a subspace, then we have
    \[ |\langle U_k,V\rangle|^2 \leq |\langle U_k',X(V)\rangle|^2.\]
    \begin{proof}
        We begin by proving the statement in case $\dim(V)=1$, so $V$ is spanned by some vector $v$. Denote the eigenvalues of $X^*X$ as $a_1\geq \ldots\geq a_d$ and the corresponding eigenvectors as $u_1,\ldots, u_d$. We notice that $Xu_1,\ldots, Xu_k$ form an orthogonal basis for $U_k'$. This holds since $(XX^*)Xu_i = X(X^*X)u_i = a_i Xu_i$, so the $Xu_i$ are indeed eigenvectors of $XX^*$ of the appropriate eigenvalues and they are orthogonal since  $\langle Xu_i, Xu_j \rangle = u_i^* X^*X u_j = a_j \langle u_i,u_j\rangle = 0$, for $i\neq j$. Thus we obtain
        \[ v=\sum_{i=1}^d  \langle u_i,v\rangle u_i, \quad \pi_{U_k}(v) = \sum_{i=1}^k \langle u_i,v\rangle u_i,\quad Xv = \sum_{i=1}^d  \langle u_i,v\rangle Xu_i ,\quad \pi_{U_k'}(Xv) = \sum_{i=1}^k  \langle u_i,v\rangle Xu_i.\]
        We further see
        \[ \|v\|^2 = \sum_{i=1}^{d} |\langle u_i,v\rangle|^2 , \quad \|\pi_{U_k}(v)\|^2 =\sum_{i=1}^k |\langle u_i,v\rangle|^2, \quad \|Xv\|^2=\sum_{i=1}^d |\langle u_i,v\rangle|^2 a_i, \quad \|\pi_{U_k'}(Xv)\|^2 = \sum_{i=1}^k |\langle u_i,v\rangle|^2 a_i.\]
        Now we obtain
        \begin{align*}
            \|\pi_{U_k}(v)\|^2 \cdot\|\pi_{U_k'^\perp}(Xv)\|^2 &= \sum_{i=k+1}^d \sum_{j=1}^k a_i |\langle u_j,v\rangle|^2\cdot |\langle u_i,v\rangle|^2\\
            &\leq \sum_{i=k+1}^d \sum_{j=1}^k a_j |\langle u_j,v\rangle|^2\cdot |\langle u_i,v\rangle|^2\\
            &= \|\pi_{U_k'}(Xv)\|^2 \cdot \|\pi_{U_k^\perp}(v)\|^2.
        \end{align*}
        Adding $\|\pi_{U_k}(v)\|^2 \cdot \|\pi_{U_k'}(Xv)\|^2$ to both sides of the inequality yields
        \[ \|\pi_{U_k}(v)\|^2\cdot \|Xv\|^2 \leq \|\pi_{U_k'}(Xv)\|^2 \cdot \|v\|^2.\]
        Finally this yields the desired inequality
        \[ |\langle U_k, V\rangle|^2 = \frac{\|\pi_{U_k}(v)\|^2}{\|v\|^2} \leq \frac{\|\pi_{U_k'}(Xv)\|^2}{\|Xv\|^2} = |\langle U_k', X(V)\rangle|^2.\]
        
        We return to the case when $\dim(V)$ is arbitrary. Denote $V':=X(V)$ and consider the restriction $X_{|V}:V\to V'$. By the singular value decomposition we can find orthonormal bases $v_1,\ldots, v_n$ of $V$ and $v_1',\ldots, v_m'$ of $V'$ such that $Xv_i = \lambda_i v_i'$ for $i\in [m]$. We denote $V_i:=\Span(v_i)$ and $V_i':=\Span(v_i')$ and see $X(V_i) = V_i'$ for $i\in [m]$. Since $Xv_i=0$ for $i>m$ we have $V_i\subseteq U_k^{\perp}$ for such $i$ and we see
        \[ |\langle U_k,V\rangle|^2 = \sum_{i=1}^n |\langle U_k,V_i\rangle|^2 = \sum_{i=1}^m |\langle U_k,V_i\rangle|^2 \leq \sum_{i=1}^m |\langle U_k',V_i'\rangle|^2 = |\langle U_k',V'\rangle|^2. \qedhere\]
    \end{proof}
\end{lemma}

\subsection{Bounding the components of the moment map}

We begin by introducing a bit of notation for various projection maps.
\[p:\bC^{d\times d} \to \bC^{d\times d}, A \mapsto A-\frac{\tr(A)}{d} I_d.\]
\[p:\bC^d \to \bC^d, a \mapsto a-\frac{1}{d} \Big(\sum_{i=1}^d a_i\Big) \1.\]
Finally we define $\Delta:\bC^d \to \bC^d$ by
\[ \Delta_i(x) :=\left\{ \begin{array}{cl}
         x_{i}-x_{i+1} & ,\text{for }\in [d-1] \\
         x_i & ,\text{for } i=d
     \end{array}\right.. \]


The aim of this section is to prove the following lemma. We will use it as follows: Any vertex $i\in Q_0$ is adjacent to at most two edges labeled $i-1,i$. Now assume we have a representation $X\in \Rep_\alpha(Q)$. Then the moment of $X$ at $i$ can be calculated as
\[ \mu_i(X) = p(A + B),\]
where $A=X_{i-1}X_{i-1}^*$ if $i-1$ points right or $A=-X_{i-1}^*X_{i-1}$ if $i-1$ points left. Similarly $B=X_iX_i^*$ if $i$ points left and $B=-X_i^*X_i$ if $i$ points right. 
In each case, we can apply one of the bounds from the next lemma.

\begin{lemma}
\label{lem:difference_of_hermitians_bound}
    Let $A,B\in \bC^{d\times d}$ be hermitian matrices with eigenvalues $a_1\geq \ldots \geq a_d$ and $b_1\geq \ldots \geq b_d$ as well as eigenvectors $u_1, \ldots, u_d$ and $w_1, \ldots, w_d$ respectively. For $k\in [d-1]$ let $U_k:=\Span(u_1,\ldots,u_k)$ as well as $W_k:=\Span(w_1,\ldots,w_k)$. Then we have
    \begin{align}
        \|p(A-B)\|^2 &\geq \frac{1}{4} \sum_{k=1}^{d-1} (\Delta_k(a)-\Delta_k(b))^2 + \frac{2}{d-1} \sum_{k=1}^{d-1} |\langle U_k,W_k^{\perp} \rangle|^2 \Delta_k(a)\Delta_k(b)\\
        \|p(A+B)\|^2 &\geq \frac{1}{4} \sum_{k=1}^{d-1} (\Delta_k(a)-\Delta_{d-k}(b))^2  + \frac{2}{d-1} \sum_{k=1}^{d-1}|\langle U_k,W_{d-k}\rangle|^2 \Delta_k(a) \Delta_{d-k}(b).
    \end{align}
\end{lemma}

The proof is based on the following Proposition, which is much stronger than \Cref{lem:difference_of_hermitians_bound} but less useful in calculations.

\begin{proposition}
\label{Lemma1}
    Let $A,B \in \bC^{d\times d}$ be hermitian matrices with eigenvalues $a_1\geq \ldots \geq a_d$ and $b_1\geq \ldots \geq b_d$ as well as eigenvectors $u_1, \ldots, u_d$ and $w_1, \ldots, w_d$ respectively. Then we have 
    \[ \| p(A-B) \|^2 = \underbrace{\|p(a-b)\|^2}_{\text{primary term}} + \underbrace{2 \sum_{i=1}^d a_ib_i - 2\sum_{i=1}^d \sum_{j=1}^d |\langle u_i,w_j \rangle|^2 a_ib_j}_{\text{secondary term}}. \]

    \begin{proof}
        Both sides of the equation are invariant under unitary transformations so we can assume that $A$ is a diagonal matrix with its eigenvalues in decreasing order on the diagonal meaning $u_i=e_i$ for $i\in [d]$. Let $\lambda_{i,j}:=|\langle e_i,w_j\rangle |^2$, then we can calculate the diagonal entries of $B$ as follows
        \[ e_i^* B e_i = \sum_{j=1}^d e_i^* B w_j w_j^* e_i = \sum_{j=1}^d b_j e_i^* w_j w_j^* e_i = \sum_{j=1}^d |\langle e_i,w_j\rangle |^2 b_j = \sum_{j=1}^d \lambda_{i,j} b_j \]
        where we used that $\sum_j w_jw_j^* = I_d$. We also know that $\sum_{i,j} |e_i^*Be_j|^2 = \|B\|^2 = \sum_i b_i^2$, thus the sum of the squares of the absolute values of the off-diagonal entries of $A-B$ is the same as the sum of the squares of the absolute values of the off diagonal entries of $B$, which is equal to
        \[ \sum_{i\neq j} |e_i^*Be_j|^2 = \sum_{i=1}^d b_i^2 - \sum_{i=1}^d \left(\sum_{j=1}^d \lambda_{i,j} b_j\right)^2 \]
        The diagonal entries of $A-B$ are $a_i-\sum_j \lambda_{i,j} b_j$, thus the sum of the squares of the absolute values of the diagonal entries of $p(A-B)$ is equal to
        \[ \sum_{i=1}^d \left(a_i-\sum_{j=1}^d \lambda_{i,j} b_j\right)^2 - \frac{1}{d} (\tr(A)-\tr(B))^2 = \sum_{i=1}^d \left( a_i^2 - 2 \sum_{j=1}^d \lambda_{i,j} a_i b_j + \left(\sum_{j=1}^d \lambda_{i,j} b_j\right)^2 \right) - \frac{1}{d} (\tr(A)-\tr(B))^2.\]
        Summing up the two terms yields the first step in the following equation
        \begin{align*}
            \|p(A-B)\|^2 &= \sum_{i=1}^d (a_i^2+b_i^2) - \frac{1}{d} (\tr(A)-\tr(B))^2 - 2 \sum_{i=1}^d\sum_{j=1}^d \lambda_{i,j} a_ib_j\\
            &= \sum_{i=1}^d (a_i-b_i)^2 +2\sum_{i=1}^d a_ib_i - \frac{1}{d} \Big(\sum_{i=1}^d (a_i-b_i)\Big)^2 - 2 \sum_{i=1}^d\sum_{j=1}^d \lambda_{i,j} a_ib_j\\
            &= \sum_{i=1}^d \Big(a_i-b_i - \frac{1}{d} \Big(\sum_{j=1}^d (a_j-b_j)\Big)\Big)^2 +2\sum_{i=1}^d a_ib_i - 2 \sum_{i=1}^d\sum_{j=1}^d \lambda_{i,j} a_ib_j\\
            &=\|p(a-b)\|^2 + 2\sum_{i=1}^d a_ib_i - 2 \sum_{i=1}^d\sum_{j=1}^d \lambda_{i,j} a_ib_j. \qedhere
        \end{align*}
    \renewcommand{\qedsymbol}{}
    \end{proof}
\end{proposition}

\begin{remark}
    We call the summand $\|p(a-b)\|^2$ the "primary term" and the other summand the "secondary term". It should be noted that both terms are non-negative. For the primary term, this is obvious and for the secondary term, this can be deduced from \Cref{rest_term_bound} below.
\end{remark}

The next two lemmata give us bounds for the primary and secondary term respectively. 

\begin{lemma}
\label{projection-spread_inequality}
    Let $a\in \bR^d$ then we have
    \[ \|p(a)\|^2 \geq \frac{1}{4} \sum_{i=1}^{d-1} \Delta_i(a)^2.\]

    \begin{proof}
        We denote $A:=\sum_{i=1}^d a_i$ and use the fact that
        \[\|p(a)\|^2 = \sum_{i=1}^d (a_i-A)^2 \geq (a_1-A)^2 + (a_d-A)^2 \geq  \frac{(a_1-a_d)^2}{4}.\]
        We further have $a_1-a_d = \sum_{i=1}^{d-1}(a_i-a_{i+1})$ and since $a_i-a_{i+1}\geq 0$ for $i\in [d-1]$ we obtain
        \[ (a_1-a_d)^2 = \Big(\sum_{i=1}^{d-1}(a_i-a_{i+1})\Big)^2 \geq \sum_{i=1}^{d-1}(a_i-a_{i+1})^2. \qedhere\]
    \end{proof}
\end{lemma}

\begin{lemma}
\label{rest_term_bound}
    Let $M=(m_{i,j})\in \bR^{d\times d}$ be doubly stochastic and let $a,b\in \bR^d$ with $a_1\geq \ldots \geq a_d$ and $b_1\geq \ldots \geq b_d$. For any index $k\in [d]$ we have
    \[ b^* (I_d-M) a \geq (a_k-a_{k+1})(b_k-b_{k+1}) \cdot \sum_{i=1}^{k} \sum_{j=k+1}^d m_{i,j}.\]
    \begin{proof}
        By the Birkhoff-von Neumann theorem \cite[Theorem 8.6]{BvN} any doubly stochastic matrix is a convex combination of permutation matrices. Since both sides of the inequality which we want to show are linear in $M$, it suffices to prove the inequality for permutation matrices, so we assume $M$ to be a permutation matrix in the rest of the proof. 
        
        We begin by showing nonnegativity of $b^*(I_d-M) a$. The linear function $b^* (I_d-M) a$ attains its minimum on the set of permutation matrices at some point $M$. We claim that the minimum is attained at $M=I_d$. Otherwise we could choose $i\in [d]$ minimal such that $m_{i,i}=0$. For such an $i$ we find indices $k,l>i$ such that $m_{k,i}=m_{i,l}=1$ which in turn implies $m_{k,l}=0$. We define $M_{i_1,j_1,i_2,j_2} := E_{i_1,j_1} + E_{i_2,j_2} - E_{i_1,j_2} - E_{i_2,j_1}$ for $i_1,j_1,i_2,j_2\in [d]$ and in case $i_1 > i_2$ and $j_1 > j_2$ we see
        \begin{align*}
             b^* M_{i_1,j_1,i_2,j_2} a &=a_{i_1}b_{j_1} + a_{i_2}b_{j_2} - a_{i_1}b_{j_2} - a_{i_2}b_{j_1} \\
            &= (a_{i_1}-a_{i_2})(b_{j_1}-b_{j_2})\geq 0. \tag{$\ast$}
        \end{align*}
        Now we can replace $M$ with $M':=M+M_{i,i,k,l}$ (one can check that the latter is still a permutation matrix) and see $b^* (I_d-M') a \leq b^* (I_d-M) a$ by $(\ast)$. Now $M'$ has ones on the diagonal entries $\{1,\ldots i\}$ and we can proceed inductively, showing that $b^* (I_d-M) a$ attains its minimum at the identity. So in particular $b^* (I_d-M) a \geq 0$ for all permutation matrices matrices $M\in \bR^{d\times d}$.

        Now we prove the inequality by induction on the integer $N:=\sum_{i=1}^{k} \sum_{j=k+1}^d m_{i,j} \in \bN$. The induction start $N=0$ is just nonnegativity, which we showed before. Now assume we have shown the inequality for $N\in \bN$ and let $M$ be a permutation matrix such that $\sum_{i=1}^{k} \sum_{j=k+1}^d m_{i,j} = N+1$. We notice that 
        \[\sum_{i=k+1}^{d} \sum_{j=1}^k m_{i,j} = k-\sum_{i=1}^{k} \sum_{j=1}^k m_{i,j} = \sum_{i=1}^{k} \sum_{j=k+1}^d m_{i,j} = N+1.\]
        Thus there exist indices $i_1,j_1 \in [k]$ and $i_2,j_2 \in \{k+1,\ldots, d\}$ such that $m_{i_1, j_2}=m_{i_2,j_1}=1$. One can check that the matrix $M':=M+M_{i_1,j_1,i_2,j_2}$ is a permutation matrix with $\sum_{i=1}^{k} \sum_{j=k+1}^d m'_{i,j}=N$ (in fact we obtain $M'$ from $M$ by switching the rows $i_1$ and $i_2$) and we obtain
        \begin{align*} 
            b^*(I_d-M)a &= b^*(I_d-M')a +b^*M_{i_1,j_1,i_2,j_2}a\\
            &= b^*(I_d-M')a + (a_{i_1}-a_{i_2})(b_{j_1}-b_{j_2}) \tag{$\ast$}\\
            &\geq b^*(I_d-M')a + (a_{k}-a_{k+1})(b_{k}-b_{k+1})\\
            &\geq (N+1)(a_{k}-a_{k+1})(b_{k}-b_{k+1}). \tag{induction hypothesis}
        \end{align*}
        This finishes the induction and we are done.
    \end{proof}
\end{lemma}

\begin{proof}[Proof of \Cref{lem:difference_of_hermitians_bound}]
        From \Cref{Lemma1} we have
        \[ \|p(A-B)\| = \|p(a-b)\|^2+ 2\sum_{i=1}^d a_{i}b_i - 2 \sum_{i=1}^d \sum_{j=1}^d |\langle u_i,w_j\rangle|^2a_ib_j =\|p(a-b)\|^2 + 2 b^*(I_d-M)a,\]
        where $M$ is the matrix with entries $m_{i,j}=|\langle u_i,w_j\rangle|^2$. Since $u_1,\ldots, u_d$ and $w_1,\ldots, w_d$ are orthonormal bases, we see that $M$ is doubly stochastic. By \Cref{rest_term_bound} we obtain
        \[ 2 b^*(I_d-M)a \geq 2(a_k-a_{k+1})(b_k-b_{k+1}) \sum_{i=1}^k \sum_{j=k+1}^{d} |\langle u_i,w_j\rangle|^2.\]
        This fact together with \Cref{projection-spread_inequality} yields the first inequality. For the second inequality observe that the spectrum $b'$ of $-B$ is $b_1'=-b_d, \ldots, b_d'=-b_1$. Thus we obtain $\Delta_k(b') = b_k'-b'_{k+1} = b_{d-k} - b_{d-k+1} = \Delta_{d-k}(b)$. Since $W_{d-k}$ is the span of the first $d-k$ eigenvectors of $B$, which is the same as the span of the last $d-k$ eigenvectors of $-B$, which is the orthogonal complement of the first $d-k$ eigenvectors of $-B$, the second inequality follows from the first one.
    \end{proof}

\subsection{Proof of the main theorem}


In this section we will prove \Cref{mainresult}. We will restrict ourselves to the case when $Q$ is of type $\hat{A}$ and treat $A$ tyope quivers by embedding them into a $\hat{A}$ type quiver. Recall our labeling of the arrows and vertices of $\hat{A}_n$
\[ Q_0,Q_1:=\bZ/n\bZ, \text{ with } (ha,ta)=\left\{ \begin{array}{cl}
    (a,a+1) & ,\text{when $a$ points right}  \\
    (a+1,a) & ,\text{when $a$ points left}
\end{array} \right. \text{ for $a\in Q_1$}.\] 
For the remainder of the section we will fix a few terms. Let $Q$ be a quiver of type $\hat{A}_n$ and let $\alpha$ be a dimension vector. Let $X\in \Rep_\alpha(Q)$ be an unstable representation of $Q$.\\
We define the \emph{head} and \emph{tail flag} as follows. Let $i\in Q_1$ index an arrow, we define
\[  \bH(X)^{i}_k \text{ is the span of the first k eigenvectors of } X_iX_i^*.\]
\[  \bT(X)^{i}_k \text{ is the span of the first k eigenvectors of } X_i^*X_i.\]
In case some singular values of $X_i$ are the same, these might not be well defined but in this case we simply choose some set of eigenvectors. Further, this case will be handled separately later. We also define the left and right flags as follows: Any vertex $i\in Q_1$ has arrows indexed by $i-1,i$ to its left and right respectively. We define
\begin{align*}
    \bL(X)_k^i &:= \left\{ \begin{array}{ll}
    \bH(X)_k^{i-1} & ,\text{ if $i-1$ points right} \\
    (\bT(X)_{\alpha(i)-k}^{i-1})^\perp & ,\text{ if $i-1$ points left}
\end{array} \right.\\
\bR(X)_k^i &:= \left\{ \begin{array}{ll}
    (\bT(X)_k^i)^\perp & ,\text{ if $i$ points right} \\
    \bH(X)_{\alpha(i)-k}^i & ,\text{ if $i$ points left} 
\end{array} \right.
\end{align*}
\begin{remark}
    Note that $\dim(\bL(X)_{k}^i) = k$ while $\dim(\bR(X)_{k}^i) = \alpha(i) - k$. Thus we have $\dim(\bL(X)_{k}^i) + \dim(\bR(X)_{k}^i) = \alpha(i)$ and we may apply \cref{lem:3_subspace_inequality} to attain bounds on $|\langle \bL(X)_{k}^i, \bR(X)_{k}^i\rangle|^2$. 
\end{remark}
Finally we define
\[\begin{array}{ll}
    s_{i,k} := \Delta_k(\spec(X_iX_i^*)) & \text{, for }i\in Q_1, k\in [\alpha_{\max}], \\
    \lambda_{i,k} := \frac{8}{\alpha_{\max}} |\langle \bL(X)_k^i, \bR(X)_k^i \rangle|^2 & \text{, for } i\in Q_0, k\in[\alpha(i)-1].
\end{array}\]
With this notation we can rephrase \Cref{lem:difference_of_hermitians_bound} as follows:

\begin{lemma}
\label{difference_of_hermitians_bound_specified}
    Let $i\in Q_0$. If $i$ is neither a source nor a sink, we have 
    \[  4\|\mu_i(X)\|^2 \geq \sum_{k=1}^{\alpha(i)-1} (s_{i-1,k}- s_{i,k})^2 + \lambda_{i,k} s_{i-1,k}s_{i,k}. \]
    In case $i$ is a source or a sink we have
    \[  4\|\mu_i(X)\|^2 \geq \sum_{k=1}^{\alpha(i)-1} (s_{i-1,k}- s_{i,\alpha(i)-k})^2 + \lambda_{i,k} s_{i-1,k}s_{i,\alpha(i)-k}. \]
\end{lemma}

In order to keep track of the summands in the last lemma, when summing over $\|\mu_i(X)\|^2$ we construct a graph $G=(V,E)$ from the quiver $Q$ and dimension vector $\alpha$. Let
\[ V:= \{ (i,k) \;|\; i\in Q_1, k\in [\alpha_{\max}]\}\]
For any vertex $i\in Q_0$ we add arrows $a_{j,k}$ for $k\in [\alpha(j)]$ as follows:\\
If $i$ is neither a sink nor a source we let $ha_{i,k}:=(i,k)$ and $ta_{i,k}:=(i-1,k)$\\
If $i$ is a sink or a source we let $ha_{i,k}:=(i,\alpha(i)-k)$ and $ta_{i,k}:=(i-1,k)$\\
To this graph we associate a function $f_{G}:\bR^V \times \bR^E \to \bR$ given by
\[f_{G}(s,\lambda) := \frac{\sum_{a\in E} (s_{ha}-s_{ta})^2 + \lambda_{a} s_{ha} s_{ta}}{(\sum_{v\in V} s_v)^2}.\]

\begin{example}
Consider the type $\hat{A}_4$ quiver with dimension vector $\alpha = (2,4,5,3)$
\[\begin{tikzcd}
    1 & 2 & 3 & 4
    \arrow["1" ,from=1-1, to=1-2]
    \arrow["4", bend left=30, from=1-1, to=1-4]
    \arrow["2"', from=1-3, to=1-2]
    \arrow["3", from=1-3, to=1-4]
\end{tikzcd}.\]
The corresponding graph $G$ looks as follows, where the zero vertices are those, such that $s_{i,k}$ vanishes for all representations $X\in \Rep_\alpha(Q)$.
\[\begin{tikzcd}
	{(1,1)} & {(2,1)} & {(3,1)} & {(4,1)} \\
	{(1,2)} & {(2,2)} & {(3,2)} & {(4,2)} \\
	0 & {(2,3)} & {(3,3)} & 0 \\
	0 & {(2,4)} & 0 & 0 \\
	0 & 0 & 0 & 0
	\arrow[from=1-1, to=3-2]
	\arrow[from=1-2, to=4-3]
	\arrow[from=1-3, to=2-4]
	\arrow[bend right=30, from=1-4, to=1-1]
	\arrow[from=2-1, to=2-2]
	\arrow[from=2-2, to=3-3]
	\arrow[from=2-3, to=1-4]
	\arrow[from=3-1, to=1-2]
	\arrow[from=3-2, to=2-3]
	\arrow[from=4-2, to=1-3]
\end{tikzcd}\]
\end{example}

With this construction we obtain the following bound on the norm of the moment map

\begin{lemma}
\label{prop:moment_graph_lower_bound}
    We have
    \[ \norm{\mu(X)}^2 \geq \frac{1}{4\alpha_{\max}^2} f_{G}(s,\lambda).\]

    \begin{proof}
        Recall the formula for the norm of the moment map from \Cref{momentmap_formula}
        \[ \|\mu(X)\|^2 = \frac{\sum_{v\in Q_0} \|\mu_v(X)\|^2}{(\sum_{a\in Q_1} \tr(X_a^*X_a))^2} = \frac{\sum_{i=1}^n \|\mu_i(X)\|^2}{(\sum_{i=1}^{n} \tr(X_i^*X_i))^2}.\]
        For the denominator we see
        \[ \sum_{i=1}^{n} \tr(X_i^*X_i) = \sum_{i=1}^{n} \sum_{k=1}^{\min(\alpha(i), \alpha(i+1))} \spec(X_i^*X_i))_k = \sum_{i=1}^n \sum_{k=1}^{\min(\alpha(i), \alpha(i+1))} k s_{i,k} \leq \alpha_{\max} \sum_{v\in V} s_v. \tag{$\ast$}\]
        For the numerator we utilize the bounds from \Cref{difference_of_hermitians_bound_specified}. Let $i\in Q_0$ be some vertex and recall that its left and right arrows are labeled $i-1,i$ respectively. In case the vertex $i$ is neither a sink nor a source, by \Cref{difference_of_hermitians_bound_specified} we have
        \[ 4 \|\mu_i(X)\|^2 \geq \sum_{k=1}^{\alpha(i)-1} (s_{i-1,k}- s_{i,k})^2 + \lambda_{i,k} s_{i-1,k}s_{i,k} = \sum_{k=1}^{\alpha(i)-1} (s_{ta_{i,k}}- s_{ha_{i,k}})^2 + \lambda_{a_{i,k}} s_{ta_{i,k}}s_{ha_{i,k}}. \]
        In case the vertex $i$ is a sink or a source, by \Cref{difference_of_hermitians_bound_specified} we have
        \[ 4 \|\mu_i(X)\|^2 \geq \sum_{k=1}^{\alpha(i)-1} (s_{i-1,k}- s_{i,\alpha(i)-k})^2 + \lambda_{i,k} s_{i-1,k}s_{i,\alpha(i)-k} = \sum_{k=1}^{\alpha(i)-1} (s_{ta_{i,k}}- s_{ha_{i,k}})^2 + \lambda_{a_{i,k}} s_{ta_{i,k}}s_{ha_{i,k}}. \]
        In both cases the sum on the right hand side corresponds exactly to the arrows $a_{i,k}\in E$ for $k\in [\alpha(i)-1]$. Summing up over all $i\in [n]$ yields
        \[ 4\sum_{i=1}^n \|\mu_i(X)\|^2 \geq f_{G}(s,\lambda) \cdot \Big(\sum_{v\in V} s_v\Big)^2. \]
        Dividing by $(\ast)$ yields the result.
    \end{proof}
\end{lemma}

    

We continue by lower bounding $f_{G}(s,\lambda)$ in terms of its connected components. The \emph{support of $s$} is defined as
\[ \supp(s) := \{ v\in V \;|\; s_v > 0 \}. \]
For a connected component $(V',E')=H\subseteq G$ we define $\lambda_{|H} \in \bR_{\geq 0}^{E'}$ and $s_{|H} \in \bR_{\geq 0}^{V'}$ as the appropriate restrictions of $\lambda$ and $s$ respectively.
We will usually write $f_{H}(s,\lambda)$ instead of $f_{H}(s_{|H},\lambda_{|H})$, since the restriction can be deduced from the index $H$.




\begin{lemma}
\label{lem:graph_component_inequality}
    Let $H_1,\ldots, H_m$ be the connected components of $G$ such that $\supp(s)\cap H_i\neq \emptyset$. If $f_{H_i}(s,\lambda)\neq 0$ for all $i\in [m]$, we have
    \[ f_{G}(s,\lambda) \geq \Big( \sum_{i=1}^m (f_{H_i}(s,\lambda))^{-1} \Big)^{-1}. \]
    
    \begin{proof}
        Denote $t_i:=\sum_{v\in V_i} s_v$. We have $t_i> 0$ if and only if $\supp(s)\cap H_i\neq \emptyset$. Since $H_1=(V_1,E_1),\ldots, H_n=(V_n,E_n)$ are the connected components of $G$ we obtain
        \[ f_{G}(s,\lambda) = \frac{\sum_{i=1}^n \sum_{e\in E_i} (s_{he}-s_{te})^2 + \lambda_{e} s_{he} s_{te}}{(\sum_{i=1}^n \sum_{v\in V_i} s_v)^2} = \frac{\sum_{i=1}^m t_i^2 \cdot f_{H_i}(s,\lambda)}{(\sum_{i=1}^m t_i)^2}.\]
        The result follows by applying \Cref{minlemma4}.
    \end{proof}
\end{lemma}

The previous lemma shows us, that it suffices to investigate the connected components of $G$. We collect some of their properties in the next lemma.

\begin{lemma}
\label{lem:G_Q_properties}
    Let $H=(V',E')$ be a connected component of $G$. It has the following properties.
    \begin{enumerate}
        \item $H$ is a graph of type $A$ or $\hat{A}$. More precisely it consists of vertices $(m,k_m), (m+1,k_{m+1}),\ldots (M,k_{M})$ subject to
        \[ k_{i+1} = \left\{ \begin{array}{ll}
         k_i & \text{, if $i+1$ is neither a source nor a sink} \\
         \alpha_{i+1}-k_i & \text{, if $i+1$ is either a source or a sink}
    \end{array} \right., \]
        with edges between adjacent vertices in the list, as well as an edge between $(M,k_M)$ and $(m,k_m)$ in case $H$ is of type $\hat{A}$.
        \item If $H$ is a singleton $(m,k)$, we have $s_{m,k}=0$.
        \item If $H$ is of type $\hat{A}$ we have $|H|=n$ and
        \[ \sum_{\substack{i=1\\i:\text{ sink}}}^n \alpha(i) = \sum_{\substack{i=1\\i:\text{ source}}}^n \alpha(i),\]
        where both sums range over the vertices of $Q$.
        \item If $H$ is of type $A$ with $E'\subseteq \supp(s)$, we have $k_m = \alpha(m)$ and $k_M=\alpha(M+1)$. In addition, we have 
        \[ \sum_{\substack{i=m+1\\i:\text{source}}}^{M} \alpha(i) - \sum_{\substack{i=m+1\\i:\text{sink}}}^{M} \alpha(i) = a \alpha(m) + b \alpha(M+1),\]
        where
        \[ a=\left\{ \begin{array}{cl}
            1 & ,\text{ if the arrow $m$ points left} \\
            -1 &,\text{ if the arrow $m$ points right}
        \end{array} \right. \quad b=\left\{ \begin{array}{cl}
            1 & ,\text{ if the arrow $M$ points right} \\
            -1 & ,\text{ if the arrow $M$ points left}
        \end{array} \right.\]
    \end{enumerate}
    \begin{proof}
        \begin{enumerate}
            \item By definition a vertex $(i,k)$ can only connect to at most one of the vertices $(i-1,k)$ or $(i-1, \alpha_i-k)$ and at most one of the vertices $(i+1,k)$ or $(i+1, \alpha_{i+1}-k)$ depending on the orientation of the arrows in $Q$. It follows, that each vertex in $H$ has at most degree two, so $H$ is of type $A$ or $\hat{A}$. The structural result holds by definition of the graph $G$.

            \item If $H$ is a singleton $(m,k)$ we know $k\geq \alpha(m+1)$, since it would otherwise have an outgoing arrow $a_{m+1,k}\in E$. Similarly we know $k\geq \alpha(m)$ since it would otherwise have an incoming arrow $a_{m,k}\in E$ or $a_{m,\alpha(m)-k}\in E$ depending on whether the vertex $m$ is a sink or a source in $Q$ or not. The matrix $X_m$ is an $\alpha(m+1)\times \alpha(m)$ matrix, thus $X_m^*X_m\in \bC^{\alpha(m)\times\alpha(m)}$ has rank at most $\min(\alpha(m),\alpha(m+1))$. In case $k>\min(\alpha(m),\alpha(m+1))$ we obtain
            \[ s_{m,k} = \Delta_k(\spec(X_m^*X_m)) = \spec_k(X_m^*X_m)-\spec_{k+1}(X_m^*X_m) = 0.\]
            The last case is when $k=\alpha(m)=\alpha(m+1)$ in which case $X_m$ is a square matrix. Then $\det(X_m)\in \bC[\Rep_\alpha(Q)]$ is an $\SL_\alpha$ invariant and since $X\in \Rep_\alpha(Q)$ is unstable we obtain $\det(X_m)=0$ and thus $\rank(X_m^*X_m)\leq k-1$. As above this yields $s_{m,k}=0$. 

            \item Assume $H$ is of type $\hat{A}$. Since there is an arrow between $(M,k_{M})$ and $(m,k_m)$ we must have $m \equiv M+1 \mod n$ as well as $k_{M+1}=k_m$. From the recursive definition of $k$ we see
            \[ k_{M+1} = k_m \pm \left(\sum_{\substack{i=m+1\\i:\text{source}}}^{M+1} \alpha(i) - \sum_{\substack{i=m+1\\i:\text{sink}}}^{M+1} \alpha(i) \right).\]
            where the $\pm$ sign depends on whether the arrow $m$ points right or left. The result follows from the fact that $k_{M+1}=k_m$ and the set $\{m+1,\ldots, M+1\}$ ranges over every vertex in $Q$.

            \item Because the vertices $(m,k_m)$ and $(M,k_M)$ are boundary vertices, we must have $k_m\geq \alpha(m)$ and $k_M\geq \alpha(M)$ since they would otherwise have an incoming and outgoing arrow respectively. By assumption we have $s_{m,k_m}, s_{M,k_M} > 0$ and therefore
            \[ \spec_{k_m}(X_m^*X_m),\spec_{k_M}(X_M^*X_M) > 0,\]
            which implies $k_m = \alpha(m)$ and $k_M = \alpha(M+1)$ respectively. We use the recursive definition of $k$ to obtain the following 4 equations
            \[ k_{M} = \left\{ \begin{array}{cl}
                k_m - \sum_{\substack{i=m+1\\i:\text{source}}}^{M} \alpha(i) + \sum_{\substack{i=m+1\\i:\text{sink}}}^{M} \alpha(i) & \text{, if $m$ and $M$ point to the left} \\
                k_m + \sum_{\substack{i=m+1\\i:\text{source}}}^{M} \alpha(i) - \sum_{\substack{i=m+1\\i:\text{sink}}}^{M} \alpha(i)& \text{, if $m$ and $M$ point to the right} \\
                - k_m + \sum_{\substack{i=m+1\\i:\text{source}}}^{M} \alpha(i) - \sum_{\substack{i=m+1\\i:\text{sink}}}^{M} \alpha(i)& \text{, if $m$ points left and $M$ points right} \\
                -k_m - \sum_{\substack{i=m+1\\i:\text{source}}}^{M} \alpha(i) + \sum_{\substack{i=m+1\\i:\text{sink}}}^{M} \alpha(i)& \text{, if $m$ points right and $M$ points left} \\
            \end{array} \right. \]
            In each case the result follows. \qedhere
        \end{enumerate}
    \end{proof}
\end{lemma}

\begin{lemma}
\label{lem:HeadTail_inequality}
    Consider a subrepresentation $Y\subseteq X$. For any arrow $i\in Q_1$ and any $k\in \bN$, such that $s_{i,k}>0$, we have
    \[|\langle \bT(X)^i_{k}, Y^{ti}\rangle|^2 \leq |\langle \bH(X)^{i}_{k}, Y^{hi}\rangle|^2.\]

    \begin{proof}
        Since $s_{i,k}>0$, we have $\spec_k(X_i^*X_i)>\spec_{k+1}(X_i^*X_i)\geq 0$. This implies $\ker(X_i)\cap \bT(X)^i_{k} = \{0\}$ so we can apply \Cref{lem:HeadTail_inequality_technique} to obtain the desired inequality.
    \end{proof}
\end{lemma}

Finally we are ready to bound $f_H(s,\lambda)$ for a connected component of $G$. For understanding the proof, we recommend focusing on the case when $H$ is of type $\hat{A}$, since it avoids much of the boundary casework.

\begin{lemma}
\label{prop:component_lower_bound}
    Let $H$ be a connected component of $G$ such that $\supp(s)\cap H\neq \emptyset$. Then the following statements hold. In case $Q$ is the oriented cycle, we have
    \[ f_{H}(s,\lambda) \geq 4 \alpha_{\max}^{-3} |H|^{-3}.\]
    In case $Q$ is of type $\hat{A}_n$ and acyclic we have
    \[ f_{H}(s,\lambda) \geq 4 \alpha_{\max}^{-2} |H|^{-3}.\]
    In case some matrix $X_i$ vanishes (this is the case when $X$ restricts to a representation of an $A_{n}$ type quiver) we have
    \[ f_H(s,\lambda) \geq 4\alpha_{\max}^{-1} |H|^{-3}. \]
    \begin{proof}
        In case $\alpha_{\max}=1$, the nullcone is empty except for zero, so we asssume $\alpha_{\max}\geq 2$. From \Cref{lem:G_Q_properties} ii) we know that $H$ cannot be a singleton since $H$ meets $\supp(s)$. By \Cref{lem:G_Q_properties} i) we know that $H$ consists of vertices $(m,k_m),\ldots, (M,k_M)$ and if any $s_{i,k_i}$ vanishes we can apply \Cref{minlemma_some_term_vanishes} to obtain
        \[ f_H(s,\lambda) \geq \frac{\sum_{i=m}^{M-1} (s_{i,k_i}-s_{i+1,k_{i+1}})^2}{(\sum_{i=m}^M s_{i,k_i})^2} \geq 3(M-m+1)^{-3} =3|H|^{-3}.\]
        We therefore assume $s_{i,k_i} > 0$ for all $i\in \{m,\ldots, M\}$. This implies that all the head and tail flags below are well defined.
        In summary, we have the following three assumptions in the remainder of the proof.
        \begin{itemize}
            \item $\alpha_{\max}\geq 2$.
            \item $H$ is not a singleton and consists of vertices $(m,k_m),\ldots, (M,k_M)$.
            \item for $i \in \{m, \ldots, M\}$ we have $s_{i,k_i} > 0$, i.e. $H \subseteq \supp(s)$.
        \end{itemize}
        \subsection*{Case 1: $H$ is of type $\hat{A}_n$ and $Q$ is not the oriented cycle}
        We begin with the case that $H$ is of type $\hat{A}_n$ and $Q$ is not the oriented cycle. We can assume w.l.o.g., that $m=1$ and $M=n$. Let $\tau$ be that weight assigning to each source the value $1$ and to each sink the value $-1$ and we see
        \[ \tau(\dim(X)) = \sum_{v \in Q_0} \tau(v) \dim(X_v) = \sum_{\substack{i=1\\i:\text{source}}}^n \alpha(i) - \sum_{\substack{i=1\\i:\text{sink}}}^n \alpha(i) = 0,\]
        by \Cref{lem:G_Q_properties} iii). Now Kings-criterion \ref{KC} implies the existence of a subrepresentation $Y\subseteq X$, such that $\tau( \dim(Y)) \geq 1$. 
        
        Let $i\in \{1,\ldots, n\}$ be some vertex. In each of the following 4 inequalities we use the definition of the left and right flags in the first step and \Cref{lem:3_subspace_inequality} in the second step. We also omit the dependence of the head and tail flags on $X$ from the notation. In case $i-1$ points left and $i$ points right , we have
        \[ \lambda_{i,k_{i-1}} =\frac{8}{\alpha_{\max}}|\langle \bT_{k_{i-1}}^{i-1}, \bT_{k_{i}}^i \rangle|^2 \geq \frac{8}{\alpha_{\max}^2} (-|\langle \bT_{k_{i-1}}^{i-1}, Y^i\rangle|^2 - |\langle Y^i, \bT_{k_{i}}^i \rangle|^2 + \dim(Y^i))^2. \]
        In case $i-1$ points right and $i$ points left, we have
        \[ \lambda_{i,k_{i-1}} = \frac{8}{\alpha_{\max}}|\langle \bH_{k_{i-1}}^{i-1}, \bH_{k_{i}}^i \rangle|^2 \geq \frac{8}{\alpha_{\max}^2} (|\langle \bH_{k_{i-1}}^{i-1}, Y^i\rangle|^2 + |\langle Y^i, \bH_{k_{i}}^i \rangle|^2 - \dim(Y^i))^2. \]
        In case $i$ and $i-1$ point right, we have
        \[ \lambda_{i,k_{i-1}} = \frac{8}{\alpha_{\max}}|\langle \bH_{k_{i-1}}^{i-1}, (\bT_{k_{i}}^i)^\perp \rangle|^2 \geq \frac{8}{\alpha_{\max}^2}(|\langle \bH_{k_{i-1}}^{i-1}, Y^i\rangle|^2 - |\langle Y^i, \bT_{k_{i}}^i  \rangle|^2)^2. \]
        In case $i$ and $i-1$ point left, we have
        \[ \lambda_{i,k_{i-1}} = \frac{8}{\alpha_{\max}}|\langle \bT_{k_{i-1}}^{i-1}, (\bH_{k_{i}}^i)^\perp \rangle|^2 \geq \frac{8}{\alpha_{\max}^2} (-|\langle \bT_{k_{i-1}}^{i-1}, Y^i\rangle|^2 + |\langle Y^i, \bH_{k_{i}}^i  \rangle|^2)^2. \]
        We notice, that for any edge $i$, the term $|\langle Y^{ti}, \bT^i_{k_i}\rangle|^2$ appears with negative sign in the inequality corresponding to the vertex $ti$ and the term $|\langle Y^{hi}, \bH^i_{k_i}\rangle|^2$ appears with positive sign in the inequality corresponding to the vertex $hi$. Furthermore the term $\dim(Y^i)$ appears with positive sign iff $i$ is a source and with negative sign iff $i$ is a sink. Thus, summing up $\lambda_{i,k_{i-1}}' := \sqrt{\frac{\alpha_{\max}^2}{8}\lambda_{i,k_{i-1}}}$ yields
        \begin{align*}
            \sum_{i=1}^{n} \lambda_{i,k_{i-1}}' &\geq \sum_{i=1}^{n} \underbrace{|\langle\bH^{i}_{k_i}, Y^{hi}\rangle|^2 - |\langle \bT^{i}_{k_i}, Y^{ti}\rangle|^2}_{\geq 0 \text{ by \Cref{lem:HeadTail_inequality}}}+ \sum_{\substack{i=1\\i:\text{source}}}^n \dim(Y^i) - \sum_{\substack{i=1\\i:\text{sink}}}^n \dim(Y^i)\\
           &\geq \sum_{\substack{i=1\\i:\text{source}}}^n \dim(Y^i) - \sum_{\substack{i=1\\i:\text{sink}}}^n \dim(Y^i)\\
           &= \tau(\dim(Y)) \geq 1,
        \end{align*}
        
        where in the second step we used \Cref{lem:HeadTail_inequality}. This now yields
        \[ \sum_{i=1}^n \lambda_{i,k_{i-1}} = \frac{8}{\alpha_{\max}^2} \sum_{i=1}^n (\lambda_{i,k_{i-1}}')^2 \geq \frac{8}{\alpha_{\max}^2n}.\]
        Now we apply \Cref{lem:minlemma5} to finish the proof
        \[ f_{H}(s,\lambda) = \frac{\sum_{i=1}^n (s_{i,k_{i}}-s_{i-1,k_{i-1}})^2 + \lambda_{i,k_{i-1}} s_{i,k_{i}}s_{i-1,k_{i-1}}}{(\sum_{i=1}^n s_{i,k_i})^2} \geq \min \left(4 \alpha_{\max}^{-2}, \frac{3}{2}\right) n^{-3} = 4 \alpha_{\max}^{-2} |H|^{-3}.\]
        \subsection*{Case 2: $H$ is of type $\hat{A}_n$ and $Q$ is the oriented cycle}
        If $Q$ is the oriented cycle, we assume w.l.o.g., that every arrow points to the right. We obtain $k_i=k_{i+1}$ for every $i\in [n]$ and we will simply denote this value as $k$. So $H$ consists of the vertices $(1,k),\ldots, (n,k)$. For every $t\in \bC$, the polynomial \[\det(X_n\cdot \ldots \cdot X_1 - tI_{\alpha(1)})\in \bC[\Rep_{\alpha}(Q)]\] is an $\SL_{\alpha}$-invariant. By the instability of $X$ we must have \[\det(X_n\cdot \ldots \cdot X_1 - tI) = \det(-tI) = (-t)^{\alpha(1)},\] thus the product $X_n\cdot \ldots \cdot X_1$ is nilpotent. Now let $v_1\in \bH_k^n$ be any vector and define $v_{i+1}:=X_iv_i$. We denote $V^i:=\Span(v_i)$ and obtain the following bound by \Cref{lem:3_subspace_inequality}
        \[ \lambda_{i,k} = \frac{8}{\alpha_{\max}}|\langle \bH_{k}^{i-1}, (\bT_{k}^i)^\perp \rangle|^2 \geq \frac{8}{\alpha_{\max}}(|\langle \bH_{k}^{i-1}, V^i\rangle|^2 - |\langle V^i, \bT_k^i\rangle|^2)^2. \]
        Similar to the argument above, we denote $\lambda_{i,k}':= \sqrt{\frac{\alpha_{\max}}{8} \lambda_{i,k}}$. From the nilpotence of $X_n \cdot \ldots \cdot X_1$, we know that $v_{n \alpha_{\max}} \in \ker(X_n)$ and since $\bT^{n}_{k}$ is spanned by eigenvectors of $X_n^*X_n$ with positive eigenvalues (since $s_{n,k}$ is assumed to be positive), we have $\bT^{n}_{k} \perp V^{n\alpha_{\max}}$ and thus
        \begin{align*}
            \sum_{i=1}^{n\alpha_{\max}} \lambda_{i,k}' &\geq |\langle\bH^{0}_{k}, V^{1}\rangle|^2 - |\langle \bT^{n\alpha_{\max}}_{k}, V^{n\alpha_{\max}}\rangle|^2 + \sum_{i=1}^{n\alpha_{\max}-1} \underbrace{|\langle\bH^{i}_{k}, V^{i+1}\rangle|^2 - |\langle \bT^{i}_{k}, V^{i}\rangle|^2}_{\geq 0 \text{ by \Cref{lem:HeadTail_inequality}}}\\
           &\geq |\langle\bH^{n}_{k}, V^{1}\rangle|^2 - |\langle \bT^{n}_{k}, V^{n\alpha_{\max}}\rangle|^2\\
           &= 1
        \end{align*}
        It follows, that
        \[ \sum_{i=1}^n \lambda_{i,k}' \geq \frac{1}{\alpha_{\max}} \]
        and we obtain
        \[ \sum_{i=1}^{n} \lambda_{i,k} \geq \frac{8}{\alpha_{\max}^3 n}.\]
        \Cref{lem:minlemma5} yields
        \[ f_{H}(s,\lambda) = \frac{\sum_{i=1}^{n} (s_{i,k}-s_{i-1,k})^2 + \lambda_{i,k} s_{i,k}s_{i-1,k}}{(\sum_{i=1}^n s_{i,k})^2} \geq \frac{4}{\alpha_{\max}^3n^3} = 4 \alpha_{\max}^{-3}|H|^{-3}.\]
        
        \subsection*{Case 3: $H$ is of type $A$}
        By \Cref{lem:G_Q_properties} we have $\alpha(m)=k_m$ and $\alpha(M+1)=k_M$. We define a weight $\tau$ as follows: For each source $i\in [n]$ we define $\tau_i:=\# \{m \leq j \leq M+1 \;|\; j\equiv i \mod n\}$ and for each sink $i\in [n]$ we define $\tau_i:=-\#\{m \leq j \leq M+1 \;|\; j\equiv i \mod n\}$. In case $m$ is neither a sink nor a source we define $\tau_m:=\pm 1$ depending on the orientation of the arrow $m$. Similarly in case $M+1$ is neither a sink nor a source we define $\tau_M:=\pm 1$ depending on the orientation of the arrow $M$. Now \Cref{lem:G_Q_properties} yields that $\tau(\dim(X))=0$ and Kings-criterion \ref{KC} yields a subrepresentation $Y\subseteq X$ with $\tau(\dim(Y))\geq 1$. We again denote $\lambda_{i,k_{i-1}}' := \sqrt{\frac{\alpha_{\max}^2}{8}\lambda_{i,k_{i-1}}}$. Here we assume that $m$ and $M$ both point right, the other three cases work similarly. We have
        \begin{align*}
            \sum_{i=m+1}^{M} \lambda_{i,k_{i-1}}' &\geq |\langle\bH^{m}_{k_m}, Y^{hm}\rangle|^2 - |\langle \bT^{M}_{k_M}, Y^{tM}\rangle|^2 + \sum_{i=m+1}^{M-1} \underbrace{|\langle\bH^{i}_{k_i}, Y^{hi}\rangle|^2 - |\langle \bT^{i}_{k_i}, Y^{ti}\rangle|^2}_{\geq 0 \text{ by \Cref{lem:HeadTail_inequality}}} + \sum_{\substack{i=m+1\\i:\text{source}}}^M \dim(Y^i) - \sum_{\substack{i=m+1\\i:\text{sink}}}^M \dim(Y^i)\\
           &\geq |\langle\bH^{m}_{k_m}, Y^{hm}\rangle|^2 - |\langle \bT^{M}_{k_M}, Y^{tM}\rangle|^2 + \sum_{\substack{i=m+1\\i:\text{source}}}^M \dim(Y^i) - \sum_{\substack{i=m+1\\i:\text{sink}}}^M \dim(Y^i) \\
           &\geq \dim(Y^m) - \dim(Y^{M+1}) + \sum_{\substack{i=m+1\\i:\text{source}}}^M \dim(Y^i) - \sum_{\substack{i=m+1\\i:\text{sink}}}^M \dim(Y^i)\\
           &= \tau(\dim(Y)) \geq 1,
        \end{align*}
        where we used the following argument in the third step. Since we have $k_m=\alpha(m)$ and $k_M=\alpha(M+1)$, we know that $\bT_{k_m}^m = \bC^{\alpha(m)}$ and $\bH_{k_M}^M = \bC^{\alpha(M+1)}$. Using \Cref{lem:HeadTail_inequality} this implies
        \begin{align*}
            |\langle\bH^{m}_{k_m}, Y^{hm}\rangle|^2 - |\langle \bT^{M}_{k_M}, Y^{tM}\rangle|^2 &\geq |\langle\bT^{m}_{k_m}, Y^{m}\rangle|^2 - |\langle \bH^{M}_{k_M}, Y^{M+1}\rangle|^2\\
            &= \dim(Y^m) - \dim(Y^{M+1}).
        \end{align*}
        As above we obtain
        \[ \sum_{i=m+1}^M \lambda_{i,k_{i-1}} \geq \frac{8}{\alpha_{\max}^2(M-m)} \geq \frac{8}{\alpha_{\max}^2 |H|}  \]
        and using \Cref{minlemma3} this yields
        \begin{align*}
            f_{H}(s,\lambda) &= \frac{\sum_{i=m+1}^M (s_{i,k_{i}}-s_{i-1,k_{i-1}})^2 + \lambda_{i,k_{i-1}} s_{i,k_{i}}s_{i-1,k_{i-1}}}{(\sum_{i=m}^M s_{i,k_i})^2}\\
            &\geq 4 \alpha_{\max}^{-2} |H|^{-3}.
        \end{align*}
        \subsection*{Case 4: $H$ is of type $A$ and some $X_i$ vanishes}
        W.l.o.g. assume that $X_n$ vanishes. We proceed exactly as in case 3, the only difference being that $Y$ splits into a direct sum of irreducible representations $Y = \bigoplus_{j} Y_j$ and since $X_n$ vanishes, $Y$ is a representation of a type $A_n$ quiver so each irreducible representation satisfies $\dim(Y_j)_{\max} = 1$. Since \[\sum_{j} \tau(\dim(Y_j)) = \tau(\dim(Y)) \geq 1,\] at least one of the summands must satisfy $\tau(\dim(Y_j)) \geq 1$. We replace $Y$ with this summand and since $\dim(Y_j)_{\max} = 1$ we may let $\lambda_{i, k_{i-1}}' := \sqrt{\frac{\alpha_{\max}}{8} \lambda_{i, k_{i-1}}}$ as in Case 2 (instead of $\lambda_{i, k_{i-1}}' := \sqrt{\frac{\alpha_{\max}^2}{8} \lambda_{i, k_{i-1}}}$)
        This improves the bound on the sum
        \[ \sum_{i=m+1}^M \lambda_{i,k_{i-1}} \geq \frac{8}{\alpha_{\max}(M-m)} \geq \frac{8}{\alpha_{\max}|H|} \]
        and by \Cref{minlemma3} 
        \[ f_{H}(s,\lambda) \geq 4\alpha_{\max}^{-1} |H|^{-3}. \qedhere\]
    \end{proof}
\end{lemma}

Finally we are ready to put everything together and prove \Cref{mainresult}.

\Mainresult*
\begin{proof}
    Let $Q$ be a quiver of type $\hat{A}_n$ and $\alpha$ a dimension vector. Let $X\in \Rep_\alpha(Q)$ be unstable. Let $H_1,\ldots, H_M$ be the connected components of $G$ ordered such that $H_1,\ldots, H_m$ are the components with $\supp(s)\cap H_i\neq \emptyset$. In case $Q$ is the oriented cycle, \Cref{prop:moment_graph_lower_bound} and \Cref{lem:graph_component_inequality} yield
    \begin{align*}
        \|\mu(X)\|^2 &\geq \frac{1}{4\alpha_{\max}^2} f_{G}(s,\lambda) \tag{ \Cref{prop:moment_graph_lower_bound}}\\
        &\geq \frac{1}{4\alpha_{\max}^2} \left( \sum_{i=1}^m f_{H_i}(s,\lambda)^{-1} \right)^{-1} \tag{\Cref{lem:graph_component_inequality}}\\
        &\geq \alpha_{\max}^{-5} \left( \sum_{i=1}^m |H_i|^{3} \right)^{-1} \tag{\Cref{prop:component_lower_bound}}\\
        &\geq \alpha_{\max}^{-6} n^{-3}. \tag{\Cref{lem:sum_of_cubes}}\\
    \end{align*}
    In the last step we used, that $|H_i| \leq n$ for any connected component $H_i$ and $\sum_{i=1}^m |H_i| \leq n \alpha_{\max}$.
    In case $Q$ is acyclic and of type $\hat{A}_n$ we obtain a different bound starting at step three.
    \begin{align*}
        \frac{1}{4\alpha_{\max}^2} \left( \sum_{i=1}^m f_{H_i}(s,\lambda)^{-1} \right)^{-1} &\geq \alpha_{\max}^{-4} \left( \sum_{i=1}^m |H_i|^{3} \right)^{-1} \tag{\Cref{prop:component_lower_bound}}\\
        &\geq \alpha_{\max}^{-4} |G|^{-3}\\
        &= \alpha_{\max}^{-7} n^{-3}.
    \end{align*}
    
    In case $Q$ is of type $A_n$ we extend $X$ to a representation of $\hat{A}_n$ by letting $X_n := 0$. It is clear that this does not change $\|\mu(X)\|$ or the fact that $X$ is unstable. Now each connected component $H_i$ of size greater than $n$ must contain a vertex of the form $(n,k)$ and for these vertices we have $s_{n,k} = 0$. We can split the component $H_i$ at the vertices of this form into smaller components, each of which has size $\leq n$ and contains a vertex with value zero. Using again the fact that $|H_i| \leq n$ for any connected component and $\sum_{i=1}^m |H_i| \leq (n+1) \alpha_{\max}$ we obtain
    \begin{align*}
        \frac{1}{4\alpha_{\max}^2} \left( \sum_{i=1}^m f_{H_i}(s,\lambda)^{-1} \right)^{-1} &\geq \alpha_{\max}^{-3} \left( \sum_{i=1}^m |H_i|^{3} \right)^{-1} \tag{\Cref{prop:component_lower_bound}}\\
        &\geq \alpha_{\max}^{-4} n^{-2}(n+1)^{-1}. \tag{\Cref{lem:sum_of_cubes}}\\
        &\geq \frac{1}{2}\alpha_{\max}^{-4} n^{-3}
    \end{align*}
    
\end{proof}


%% file: sigma_semistability.tex
Let $Q$ be an acyclic quiver with $n$ vertices, $\alpha$ a dimension vector and $\sigma :Q_0 \to \bZ$ a weight. Recall the $\GL_\alpha$ representation $\Rep_\alpha(Q) \oplus \chi_\sigma$, 
discussed in the introduction and the fact, that a representation $V\in \Rep_\alpha(Q)$ is $\sigma$-semistable, when $(V,1) \in \Rep_\alpha(Q) \oplus \chi_\sigma$ is $\GL_\alpha$-semistable. In this section we lowerbound the gap for this group action.

\begin{restatable}{theorem}{sigmaMain}
    \label{mainresult_sigma_semistability}
    The gap of $\Rep_\alpha(Q)\oplus \chi_\sigma$ satisfies
    \[ \gamma_{\GL}(Q,\alpha,\sigma) \geq \sqrt{\frac{4}{\alpha_{\max} n^3 (\langle |\sigma|,\alpha\rangle +2)^2}} \geq \Omega(\alpha_{\max}^{-1.5} n^{-2.5} |\sigma|_{\max}^{-1}).\]
\end{restatable}

This shows, that \cite[Algorithm 4.2]{Burgisser_2019} solves the nullcone problem for this group action in time polynomial in $n, \alpha_{\max}, \sigma_{\max}$ and the bitlength of the entries of our representation. In particular this yields another proof for \cite[Theorem 1.1]{algorithmic_sigma}. We also show that the dependence on $|\sigma|_{\max}$ is unavoidable. Since $|\sigma|_{\max}$ is exponential in the bit length of $\sigma$ this suggests that the first order algorithm \cite[Algorithm 4.2]{Burgisser_2019} doesn't in general solve the nullcone membership problem in polynomial time for this group action. We begin our analysis by calculating the moment map.

\begin{proposition}
\label{moment_map_formula_sigma}
    The moment map $\mu : \Rep_\alpha(Q)\oplus \chi_\sigma \to \bigoplus_{v\in Q_0} \Herm_{\alpha(v)}$ is given by
    \[ \mu(X,y) = \frac{1}{\sum_{a\in Q_1}\| X_a \|^2 + |y|^2} \left(\sigma(v) |y|^2 I_{\alpha(v)}  + \sum_{\substack{a\in Q_1 \\ ha=v}} X_aX_a^* - \sum_{\substack{a\in Q_1 \\ ta=v}} X_a^*X_a \right)_{v\in Q_0}.\]

    \begin{proof}
        We begin by calculating the moment map $\mu':\bC \to \bigoplus_{v\in Q_0} \Herm_{\alpha(v)}$ of $\chi_\sigma$. The associtated Lie-Algebra action is
        \begin{align*}
            \Pi'(H) &= \partial_{t=0} \prod_{v\in Q_0} \det(e^{tH_v})^{\sigma(v)}\\ &= \partial_{t=0} \exp \left(t \sum_{v\in Q_0} \sigma(v)\tr(H_v)\right)\\ &= \sum_{v\in Q_0} \sigma(v)\tr(H_v).
        \end{align*}
        Thus we obtain $\mu'(y) = (\sigma(v) I_{\alpha(v)})_{v\in Q_0}$ since this is the unique element satisfying
        \[ \tr(\mu'(y) H) = \frac{\langle y, \Pi'(H)y \rangle}{|y|^2} = \sum_{v\in Q_0} \sigma(v)\tr(H_v).\]
        Combining this with \Cref{momentmap_formula} yields the result, since
        \[ \mu(X,y) = \frac{1}{\sum_{a\in Q_1}\| X_a \|^2 + |y|^2} \left(\sigma(v) |y|^2 I_{\alpha(v)}  + \sum_{\substack{a\in Q_1 \\ ha=v}} X_aX_a^* - \sum_{\substack{a\in Q_1 \\ ta=v}} X_a^*X_a \right)_{v\in Q_0}\]
        is the unique element satisfying
        \begin{align*}
            \tr(\mu(X,y) H) = \frac{\langle(X,y),\Pi(H)(X,y)\rangle}{\| (X,y) \|^2} = \frac{\sum_{v\in Q_0} \tr\left(\left( |y|^2 \sigma(v) I_{\alpha(v)} + \sum_{\substack{a\in Q_1 \\ ha=v}} X_aX_a^* - \sum_{\substack{a\in Q_1 \\ ta=v}} X_a^*X_a \right)H_v\right)}{\sum_{a\in Q_1} \|X_a\|^2 + |y|^2}.
        \end{align*}
    \end{proof}
\end{proposition}

\begin{definition}
    Let $U\subseteq \bC^d$ be a subspace and $X:\bC^d\to \bC^d$ a linear map. The rayleightrace is defined as
    \[ \tr_U(X) := \tr(\pi_U \circ X \circ \iota_U), \]
    where $\iota_U:U\to \bC^d$ is the inclusion map and $\pi:\bC^d\to U$ is the orthogonal projection onto $U$.
\end{definition}

\begin{lemma}
\label{rayleigh_trace_increases_under_map}
    Let $X\in \bC^{d',d}$ and $U\subseteq \bC^d$ a subspace. Then we have
    \[ \tr_U(X^*X) \leq \tr_{X(U)}(XX^*).\]
    \begin{proof}
        Denote $V:= X(U)$ and consider $X_{|U}:U\to V$. The singular value decomposition yields unitary bases $u_1,\ldots, u_{\dim(U)}$ and $v_1,\ldots, v_{\dim(V)}$, such that $Xu_i = \lambda_i v_i$ for $i\leq \dim(V)$ and $Xu_i=0$ for $i>\dim(V)$. For any $i \in [\dim(V)]$ we obtain using the Cauchy-Schwarz inequality that 
        \[ u_i^* X^* X u_i = \lambda_i^2 = (v_i^*Xu_i)^2 \leq \|X^*v_i\|^2 = v_i^* X X^* v_i. \]
        Thus we obtain
        \[ \tr_U(X^*X) = \sum_{i=1}^{\dim(V)} u_i^* X^* X u_i \leq \sum_{i=1}^{\dim(V)} v_i^* X X^* v_i = \tr_V(XX^*). \qedhere \]
    \end{proof}
\end{lemma}

\begin{lemma}
\label{frobenius_trace_bound}
    Let $A\in \bC^{d\times d}$ be hermitian and $U\subseteq \bC^{d}$ be any subspace. Then
    \[ \|A\|^2 \geq \frac{\tr_U(A)^2}{\dim(U)} \]
    \begin{proof}
        By observing $\| \pi_U \circ A\circ \iota_U\| \leq \|A\|$ we may assume $U=\bC^d$. Let $\lambda$ be the spectrum of $A$, then we see by the Cauchy-Schwarz inequality, that
        \[ \tr(A) = \sum \lambda_i = \langle \lambda, \1 \rangle \leq \|\lambda\| \cdot \|\1\| = \sqrt{d}\|A\|.\qedhere \] 
    \end{proof}
\end{lemma}

\begin{lemma}
\label{graph_flow_lemma}
    Let $Q$ be an acyclic quiver with non-negative edge weights $x\in \bR_{\geq 0}^{Q_1}$. We denote the flow on $Q$ by $F:Q_0\to \bR, v\mapsto \sum_{\substack{a\in Q_1 \\ ha=v}} x_a - \sum_{\substack{a\in Q_1 \\ ta=v}} x_a$. Then we have
    \[ \sum_{v\in Q_0} |F(v)| \geq \frac{2}{n} \sum_{a\in Q_1} x_a. \]

    \begin{proof}
        W.l.o.g. we assume that $Q$ is the complete acyclic digraph on $n$ vertices. Now assume there is an arrow $(i,j)$ with $x_{i,j}>0$ and $j>i+1$. Then we can pass to $x'_{i,j}:=0$ and $x'_{k,k+1}:= x_{k,k+1} + x_{i,j}$ for $i\leq k<j$. Under this replacement, the flow doesn't change, while the R.H.S. of our inequality increases, thus we can assume that $Q$ is the equioriented $A_n$ quiver. Now assume, that there exists an index $i\neq 1$, such that $F(i)<0$, then we can pass to $x_{i-1,i}':= x_{i-1,i}+F(i)$. This does not decrease the L.H.S., while increasing the R.H.S. Thus we can assume $F(i)\geq 0$ for $i>1$. Similarly we can assume $F(i)\leq 0$ for $i<n$ and we obtain $F(i)=0$ for $1<i<n$. Finally this yields $x_{1,2}=x_{2,3}=\ldots= x_{n-1,n}$ and the inequality which we wanted to show holds with equality.
    \end{proof}
\end{lemma}

Now we can prove the main theorem of this section.

\sigmaMain*
\begin{proof}
        Let $(X,y) \in \Rep_\alpha(Q)\oplus \chi_\sigma$ be unstable. In case $y = 0$ we have $\mu(X, 0) = \mu_{\GL}(X)$, where the moment map on the right, comes from interpreting $X$ as an element of $\Rep_\alpha(Q)$ under the $\GL_\alpha$ action. Using \cite[Theorem 6.21]{Burgisser_2019} we can lowerbound the weight margin of that representation as
        \[ \|\mu_{\GL}(X)\| \geq \gamma_{\GL}(Q, \alpha) \geq \gamma_{T}(Q, \alpha) \geq \Big(\sum_{v\in Q_0} \alpha(v))\Big)^{-\frac{3}{2}} \geq \alpha_{\max}^{-1.5} n^{-1.5}.\]
        In case $y\neq 0$ we can assume w.l.o.g., that $y=1$ (by passing to $(y^{-1}X,1) \in \Rep_\alpha(Q)\oplus \chi_\sigma$). We denote $x_a := \tr(X_aX_a^*) = \tr(X_a^*X_a)$ for $a\in Q_1$ and the corresponding flow on $Q$ as $F:Q_0\to \bR$. We also define
        \[ \Delta := \sum_{v\in Q_0} |F(v)| -\langle |\sigma|,\alpha\rangle = \sum_{v\in Q_0} |F(v)| - \sum_{v\in Q_0} |\sigma(v)|\alpha(v).\]
        We recall the formula for the moment map \Cref{moment_map_formula_sigma} and calculate
        \[ \|\mu(X,1)\|^2 = \frac{\sum_{v\in Q_0} \|\mu_v(X,1)\|^2}{(\sum_{a\in Q_1} x_a + 1)^2}.\]
        For the denominator of the moment map we see by using \Cref{graph_flow_lemma} in the first step, that
        \[\Big(\sum_{a\in Q_1} x_a + 1\Big)^2
            \leq \Big(\frac{n}{2} \sum_{v\in Q_1} |F(v)| + 1 \Big)^2
            \leq \frac{n^2}{4} (\langle |\sigma|,\alpha\rangle +\Delta +1)^2 \tag{$\ast$} \]
        In case $ \Delta \geq 1$ or $\sigma(\alpha) \neq 0$, we lowerbound the numerator of the moment map by using \Cref{frobenius_trace_bound} in the second step
        \begin{align*}
            \|\mu_v(X,1)\|^2&= \left\| \sigma(v) I_{\alpha(v)}  + \sum_{\substack{a\in Q_1 \\ ha=v}} X_aX_a^* - \sum_{\substack{a\in Q_1 \\ ta=v}} X_a^*X_a \right\|^2\\ &\geq \frac{1}{\alpha_{\max}} \tr\left( \sigma(v) I_{\alpha(v)}  + \sum_{\substack{a\in Q_1 \\ ha=v}} X_aX_a^* - \sum_{\substack{a\in Q_1 \\ ta=v}} X_a^*X_a \right)^2\\ &= \frac{1}{\alpha_{\max}} \left(\sigma(v) \alpha(v) + \sum_{\substack{a\in Q_1 \\ ha=v}} x_a - \sum_{\substack{a\in Q_1 \\ ta=v}} x_a\right)^2
            \\ &= \frac{1}{\alpha_{\max}} \left(\sigma(v) \alpha(v) + F(v)\right)^2.
        \end{align*}
        In case $\sigma(\alpha) \neq 0$, summing up over $v\in Q_0$ yields
        \[\sum_{v\in Q_0} (\sigma(v)\alpha(v) + F(v))^2 \geq \frac{1}{n} \Big(\sum_{v\in Q_0} \sigma(v)\alpha(v) + F(v)\Big)^2 = \frac{1}{n} (\sigma(\alpha))^2 \geq \frac{1}{n}. \]
        In case $\Delta \geq 1$, we obtain a similar bound
        \begin{align*}
            \sum_{v\in Q_0} (\sigma(v)\alpha(v) + F(v))^2 \geq \sum_{v\in Q_0} (|\sigma(v)\alpha(v)| - |F(v)|)^2 \geq \frac{1}{n}(\sum_{v\in Q_0} |\sigma(v)\alpha(v)| - |F(v)|)^2 = \frac{\Delta^2}{n} \geq \frac{1}{n}.
        \end{align*}
        Combining this with $(\ast)$ yields
        \[ \|\mu(X,1)\|^2 \geq \frac{4}{\alpha_{\max} n^3 (\langle |\sigma|,\alpha\rangle +2)^2}.\]
        Now assume $\Delta<1$ and $\sigma(\alpha) = 0$. By Kings-criterion \ref{KC}, there exists a subrepresentation $Y\subseteq X$, with dimension vector $\beta$, such that $\langle \sigma, \beta\rangle \geq 1$. Now we can apply \Cref{frobenius_trace_bound}, to see
        \begin{align*}
            \|\mu_v(X,1)\|^2 &= \left\| \sigma(v) I_{\alpha(v)}  + \sum_{\substack{a\in Q_1 \\ ha=v}} X_aX_a^* - \sum_{\substack{a\in Q_1 \\ ta=v}} X_a^*X_a \right\|^2\\ &\geq \frac{1}{\alpha_{\max}} \tr_{Y(v)} \left( \sigma(v) I_{\alpha(v)}  + \sum_{\substack{a\in Q_1 \\ ha=v}} X_aX_a^* - \sum_{\substack{a\in Q_1 \\ ta=v}} X_a^*X_a \right)^2\\ &= \frac{1}{\alpha_{\max}} \left(\sigma(v) \beta(v) + \sum_{\substack{a\in Q_1 \\ ha=v}} \tr_{Y(v)}(X_aX_a^*) - \sum_{\substack{a\in Q_1 \\ ta=v}} \tr_{Y(v)}(X_a^*X_a)\right)^2.
        \end{align*}
        By applying \Cref{rayleigh_trace_increases_under_map} in the last step, we see
        \begin{align*}
            &\sum_{v\in Q_0} \Big(\sigma(v) \beta(v) + \sum_{\substack{a\in Q_1 \\ ha=v}} \tr_{Y(v)}(X_aX_a^*) - \sum_{\substack{a\in Q_1 \\ ta=v}} \tr_{Y(v)}(X_a^*X_a)\Big)\\
            &= \langle \sigma, \beta\rangle + \sum_{a\in Q_1} \tr_{Y(ha)}(X_aX_a^*) - \tr_{Y(ta)}(X_a^*X_a)\\
            &\geq 1.
        \end{align*}
        Summing up over $v\in Q_0$ and combining this with $(\ast)$ yields
        \[  \|\mu(X,1)\|^2 = \frac{\sum_{v\in Q_0} \|\mu_v(X,1)\|^2}{(\sum_{a\in Q_1} x_a + 1)^2} \geq \frac{1}{n \alpha_{\max}} \frac{4}{n^2 (\langle |\sigma|,\alpha\rangle +\Delta +1)^2} \geq \frac{4}{\alpha_{\max} n^3 (\langle |\sigma|,\alpha\rangle +2)^2}.\]
        Thus we have shown
        \[ \gamma_{\GL}(Q,\alpha,\sigma) \geq \sqrt{\frac{4}{\alpha_{\max} n^3 (\langle |\sigma|,\alpha\rangle +2)^2}}\geq \Omega(\alpha_{\max}^{-1.5} n^{-2.5} |\sigma|_{\max}^{-1}). \qedhere\]
    \end{proof}

Next we show with an example, that the dependence on $|\sigma|_{\max}$ is unavoidable even when $\sigma(\alpha) = 0$. Let $Q_n$ be obtained from an alternating $A_n$ quiver by replacing each arrow with two copies, for example $Q_4$ is
\[\begin{tikzcd}
	\bullet & \bullet & \bullet & \bullet
	\arrow[shift right, from=1-1, to=1-2]
	\arrow[shift left, from=1-1, to=1-2]
	\arrow[shift right, from=1-3, to=1-2]
	\arrow[shift left, from=1-3, to=1-2]
	\arrow[shift left, from=1-3, to=1-4]
	\arrow[shift right, from=1-3, to=1-4]
\end{tikzcd}\]
We label the vertices as $1$ to $n$ starting on the left and we similarly label the arrows $1,1',2,2',\ldots$.

\begin{proposition}
    Let $n$ be even and let $\alpha=(2,3,3,\ldots,3,1)$ be a dimension vector for $Q_{n+1}$ and let $\sigma=(-1,2,-4,8,\ldots)$ be a weight. We have $\sigma(\alpha)=0$ and
    \[ \gamma_{\GL}(Q_n,\alpha,\sigma) \leq \sqrt{\frac{5}{6}} \cdot \frac{1}{2^{n+2}-3} \leq \mathcal{O}(|\sigma|_{\max}^{-1}). \]
\begin{proof}
    It is straightforward to check that $\sigma(\alpha)=0$. We define a representation $X\in \Rep_\alpha(Q_n)$ by   
    \[ X_1 = \begin{pmatrix}
    1 & 0 \\ 0 & 0 \\ 0 & 0
    \end{pmatrix} \quad X_1' = \begin{pmatrix}
    0 & 0 \\ 0 & 0 \\ 0 & 0
    \end{pmatrix}, X_n = \begin{pmatrix}
    2^{\frac{n}{2}} \\ 0 \\ 0
    \end{pmatrix} \quad X_n' = \begin{pmatrix}
    0  \\ 2^{\frac{n}{2}}  \\ 0 
    \end{pmatrix} \]
    as well as 
    \[ \begin{pmatrix}
    X_{2k} \\ X_{2k}' \\
    X_{2k+1} \\ X_{2k+1}' \end{pmatrix} = \begin{pmatrix}
    2^{\frac{k}{2}} E_{2,1} \\ 2^{\frac{k}{2}} E_{3,1} \\
    2^{\frac{k+1}{2}} E_{1,2} \\ 2^{\frac{k+1}{2}}E_{1,3} \end{pmatrix}\]
    for $k\in [\frac{n}{2}-1]$.
    
    We see, that $X$ is $\sigma$-unstable, since it has a subrepresentation $Y\subseteq X$ with dimension vector $\beta = (1,3,3,\ldots, 3,1)$. We then have $\sigma(\beta) = \sigma(\alpha) + 1 = 1$ and the instability follows from \Cref{KC}.
    
    We immediately see $\mu_{n+1}(X)=0$, since $\alpha(n+1)=1$. For $k\geq 1$ we see
    \begin{align*}
        \mu_{2k+1}(X) &= -p(X_{2k}^* X_{2k} + (X'_{2k})^* X_{2k}' + X_{2k+1}^* X_{2k+1} + (X'_{2k+1})^* X_{2k+1}')\\
        &= -p(2^{k} E_{1,1} + 2^{k} E_{1,1} + 2^{k+1} E_{2,2} + 2^{k+1} E_{3,3})\\
        &= -2^{k+1} p(I_3) = 0.
    \end{align*}
    And similarly for $k\geq 2$
    \begin{align*}
        \mu_{2k}(X) &= p(X_{2k} X_{2k-1}^* + X_{2k-1}'(X'_{2k-1})^* + X_{2k}X_{2k}^* + X_{2k}'(X'_{2k})^*)\\
        &= p(2^{k-1} E_{1,1} + 2^{k-1} E_{1,1} + 2^{k} E_{2,2} + 2^{k} E_{3,3})\\
        &= -2^{k} p(I_3) = 0.
    \end{align*}
    We also see
    \[ \|\mu_1(X)\|^2 = \|\diag(\frac{1}{2}, -\frac{1}{2})\|^2 = \frac{1}{2}.\]
    \[ \|\mu_2(X)\|^2 = \|\diag(\frac{2}{3}, -\frac{1}{3}, -\frac{1}{3})\|^2 = \frac{2}{3}. \]
    Thus we obtain $\| \Tilde{\mu}(X)\|=\sqrt{\frac{5}{6}}$. And another computation shows
    \[ \|X\|^2 = -1 + 2\sum_{k=1}^{n} 2^k = 2^{n+2}-3.\]
    Finally we obtain
    \[ \gamma_{\GL}(Q_n,\alpha,\sigma) \leq \|\mu(X)\| = \sqrt{\frac{5}{6}} \cdot \frac{1}{2^{n+2}-3} \leq \mathcal{O}(2^{-n}). \qedhere\]
\end{proof}
\end{proposition}

We should note that the representation we constructed in the proof above is almost (up to changing $X_1'$) isomorphic to the representation constructed in \cite[Proposition 1.5]{derksen2016degreeboundssemiinvariantrings}. It is also similar to the construction in \cite[Theorem 4.25]{FR21}.

%% file: weightmargin_and_gap_star_quivers.tex
In this section we prove that for any tree quiver with $n$ leaves, there exists a dimension vector $\alpha$, whose entries are upper bounded by $n+1$, such that
\[ \gamma_{\SL}(Q,\alpha) \leq \mathcal{O}(2^{-\frac{n}{2}}).\]
In particular this shows, that our results from the previous sections cannot be generalized to all tree quivers. We begin by showing this for star quivers which are defined as follows. An $n$-star quiver is a quiver with
\[ Q_0:=\{0,1\ldots, n\} \quad Q_1:=\{1,\ldots, n\}, \]
with $(ha, ta) \in \{(0, a),(a,0)\}$ for $a\in Q_1$. For example, the equioriented $4$-star quiver is
\[\begin{tikzcd}
	& 2 \\
	1 & 0 & 3 \\
	& 4
	\arrow[from=1-2, to=2-2]
	\arrow[from=2-1, to=2-2]
	\arrow[from=2-3, to=2-2]
	\arrow[from=3-2, to=2-2]
\end{tikzcd}\]
For a finite set $A\subseteq \bR^n$ the \emph{margin} is defined as
\[ \gamma(A) : = \min\{ \dist(0, \conv(B)) \;|\; B\subseteq A, \; 0\not\in \conv(B) \}. \]
Recall, that $p:\bC^d \to \bC^d, x\mapsto x- \frac{\langle x, \1\rangle}{d} \1$ is the projection map along the ones vector.

\begin{proposition}
\label{weightmargin_star_quiver}
    Let $Q$ be an $n$-star-quiver. Then there exists a dimension vector $\alpha$, with $\alpha(0)=n+1$ and $\alpha(k)\leq n$ for $k\in [n]$, such that
    \[ \gamma_{\SL}(Q,\alpha) \leq \gamma(p(\{0,1\}^{n+1})) \leq \frac{1}{2^{\lfloor \frac{n+1}{2} \rfloor}-1}.\]

    \begin{proof}
        We only prove the first inequality, the second one is treated in \Cref{margin_of_projected_cube} below. For the time being we assume that $Q$ is equioriented, i.e. all arrows point toward the central vertex. Let $u_1,\ldots, u_{n} \in \{0,1\}^{n+1}$ and $\lambda\in \bR_{\geq 0}^n$ realize the weight margin $\gamma(p(\{0,1\}^{n+1}))$, i.e. $0\not\in \conv(p(u_1),\ldots, p(u_n))$ and
        \[ \Big\| \sum_{k=1}^n \lambda_k p(u_k)  \Big\| = \gamma(p(\{0,1\}^{n+1})), \quad \sum_{k=1}^{n} \lambda_k =1.\]
        Let $\alpha\in \bN^{Q_0}$ be defined by $\alpha(0):=n+1$ and $\alpha(k):=\sum_{i=1}^{n+1} u_{k,i} = \#\supp(u_k)$. Let $v_k\in \bN^{\alpha(k)}$ be defined by
        \[ v_{k,i} := \text{ the location of the $i$-th one in $u_k$}. \]
        Let $X\in \Rep_\alpha(Q)$ be defined by $ X_k:\bC^{\alpha(k)} \to \bC^{n+1}, e_i \mapsto \sqrt{\lambda_k} e_{v_{k,i}}$. This ensures that $X_k^*X_k = \lambda_k I_{\alpha(k)}$ and $X_kX_k^* = \lambda_k \diag(u_k)$. By \Cref{momentmap_formula}, it follows that $\mu_k(X) = 0$ for $k\in [n]$ and $\mu_0(X) = p(\diag( \sum \lambda_k u_k ))$.
        For the denominator of the moment map we see
        \[ \|X\|^2 =\sum_{k=1}^n \tr(X_k^*X_k) = \sum_{k=1}^n \lambda_k \alpha(k) \geq \sum_{k=1}^n \lambda_k = 1.\]
        Combining these observations yields
        \[ \|\mu(X)\| = \frac{\sqrt{ \sum_{k=0}^n \| \mu_k(X) \|^2 }}{\sum_{k=1}^n \tr(X_k^*X_k)} \leq \|\mu_0(X)\| = \Big\|\sum_{k=1}^n \lambda_k p(u_k)\Big\| = \gamma(p(\{0,1\}^{n+1})). \]
        It remains to show, that $X$ is $\SL_{\alpha}$-unstable. Indeed we will show, that $X$ is $\ST_\alpha$-unstable. Since $0\not\in \conv(p(u_1),\ldots, p(u_n))$, there exists a vector $w\in \bR^{n+1}$ with $\sum_{k=1}^{n+1} w_k =0$, such that $\langle w,u_k\rangle = \langle w,p(u_k)\rangle >0$ for all $k\in [n]$. Now let $M(t)\in \SL_\alpha$ be defined by \[M_0(t):=\diag(t^{w_1},\ldots, t^{w_{n+1}})\]
        and 
        \[M_k(t) :=t^{-\frac{\langle w,u_k\rangle}{\alpha(k)}} \cdot \diag(t^{w_{v_{k,1}}},\ldots, t^{w_{v_{k,\alpha(k)}}}).\]
        One can check that the determinants of these matrices are one and
        \[ \lim_{t\to 0} (M(t)\cdot X)_k = \lim_{t\to 0} t^{\frac{\langle w,u_k\rangle}{\alpha(k)}} X_k = 0,\]\
        for all $k\in[n]$, thus $X$ is unstable as claimed.
        
        Now we return to the case when $Q$ has arbitrary orientation. Consider an equioriented $n$-star quiver $Q'$, then by the above construction, we get an unstable representation $X\in \Rep_{\alpha}(Q')$. Each map $X_k:\bC^{\alpha(k)} \to \bC^{n+1}$ is an isometric embedding and we let $X_k': \bC^{n+1}\to \bC^{n+1-\alpha(k)}$ be $\sqrt{\lambda_k}$ times the orthogonal projection of $\bC^{n+1}$ onto $\im(X_k)^{\perp}$. This yields $X_k'(X_k')^* = \lambda_k I_{n+1-\alpha(k)}$ and $(X_k')^*X_k'=\lambda_k I_{n+1} -X_kX_k^*$. Now we define a representation $Y\in \Rep(Q)$ as follows. For each arrow $k\in Q_0$ we let
        \[ Y_k :=\left\{ \begin{array}{cc}
            X_k &, \text{when } (hk,tk) = (0,k)\\
            X_k' &, \text{when } (hk,tk) = (k,0)
        \end{array} \right. \]
        As above we obtain $\mu_k(Y) = 0$ for $k\in [n]$ and $\mu_0(Y) = \mu_0(X)$, as well as
        \[ \|Y\|^2 = \sum_{k=1}^n \tr(Y_k^*Y_k) \geq \sum_{k=1}^n \lambda_k = 1.\]
        Together this yields
        \[ \|\mu(Y)\| = \frac{\sqrt{\sum_{k=0}^n \|\mu_0(Y)\|^2}}{\|Y\|^2} \leq \|\mu_0(Y)\| = \gamma(p(\{0,1\}^{n+1})).\]
        The instability of $Y$ follows in a similar way to that of $X$.
    \end{proof}
\end{proposition}

\begin{lemma}
\label{margin_of_projected_cube}
    The margin of the projected unit cube $p(\{0,1\}^n)$ is exponentially small
    \[ \gamma(p(\{0,1\}^n)) \leq \frac{1}{2^{\lfloor \frac{n}{2} \rfloor}-1} \]
    \begin{proof}
        W.l.o.g. we may assume $n$ even. We write $n-1=2k+1$. Consider the sets of vectors 
        \[V:=\{ (1,0),(0,1) \}^k\times \{0\} \quad W:=\{ (1,0)^{i-1} (0,0) (1,1)^{k-i}  (1) \;|\; i\in [k]\}.\]
        Assume towards a contradiction, that $\Span(V\cup W) \neq \bR^{2k+1}$, then there exists a non zero $\mu\in \bR^{2k+1}$, such that $\langle \mu, v\rangle =0$ for all $v\in \Span(V\cup W)$. Since for all $i\in [k]$ we have
        \[ (0,0)^{i-1}(1,-1)(0,0)^{k-i}(0) = (1,0)^{i}(0,1)^{k-i}(0) - (1,0)^{i-1}(0,1)^{k-i+1}(0) \in \Span(V\cup W),\]
        we see that $\mu$ has the form $\mu = (\mu_k,\mu_k,\mu_{k-1},\mu_{k-1},\ldots, \mu_1,\mu_1, \mu_0)$. The vectors in $V$ yield the equation
        \[ 0= \sum_{i=1}^k \mu_i.\]
        The vectors in $W$, further yield the equations
        \[ 0 = \mu_0 + 2\sum_{i=0}^{j-1} \mu_i + \sum_{i=j+1}^k \mu_i = - \mu_j + \sum_{i=0}^{j-1} \mu_i + \sum_{i=1}^k \mu_i = - \mu_j + \sum_{i=0}^{j-1} \mu_i, \]
        for $j\in[k]$. These equations inductively yield $\mu_j = 2^{j-1} \mu_0$ for $j\in[k]$. But now the first equation simplifies to 
        \[ 0 = \sum_{i=1}^k \mu_i = (2^{k}-1) \mu_0\]
        and thus $\mu_0=0$, which implies $\mu=0$ the desired contradiction. Thus we have obtained $\Span(V\cup W) = \bR^{2k+1}$ and we can construct an invertible matrix $A\in \bR^{(n-1)\times (n-1)}$ with rows $r_1,\ldots, r_{n-1} \in V\cup W$.
        
        Let $\lambda:=(2^{k-1},2^{k-1},2^{k-2},2^{k-2},\ldots, 1,1,1)$. For each $v\in V\cup W$, we see $\langle v,\lambda\rangle = 2^k-1$, which means $p(A \lambda^T) = 0$ and since $A$ is invertible and the kernel of $p$ is one dimensional, we obtain $\ker(p\circ A) = \bR \lambda$. Now let $A'\in \{0,1\}^{n\times (n-1)}$ be the matrix obtained by appending the row $w:=(1,0)^k(1)$ to $A$. We see $\langle \lambda, w\rangle = 2^k$ and thus $p'(A' \lambda^T) \neq 0$, where $p':\bR^{n} \to \bR^{n}$ is the usual projection. This yields that $p'\circ A'$ is injective, since $\bR\lambda$ is the only candidate for the kernel by the previous argument. Denote the columns of $A'$ by $u_1,\ldots, u_{n-1} \in \{0,1\}^n$. By the injectivity of $p'\circ A'$, we see that $0\not\in \conv\{ p(u_1),\ldots, p(u_{n-1}) \}$ but we also see
        \[ \sum_{i=1}^{n-1} \lambda_i p(u_i) = p\begin{pmatrix}
            2^k-1\\ \ldots \\ 2^{k}-1\\ 2^k
        \end{pmatrix}= \frac{1}{n} \begin{pmatrix}
            -1\\ \ldots \\ -1\\ n-1
        \end{pmatrix}.\]
        Thus we obtain $\frac{1}{n (2^{k+1}-1)} (-1,\ldots, -1,n-1)^T \in \conv\{ p(u_1),\ldots, p(u_{n-1}) \}$. This vector has norm $\frac{\sqrt{n(n-1)}}{n (2^{k+1}-1)} \leq \frac{1}{2^{k+1}-1}$ and the result follows.
    \end{proof}
\end{lemma}

Next we generalize this result to arbitrary quivers with $n$ leaves.

\begin{theorem}
\label{tree_quivers_leaf_bound}
    Let $Q$ be a tree quiver with $n$ leaves. Then there exists a dimension vector $\alpha$ with $\alpha_{\max} = n+1$, such that \[ \gamma_{\SL}(Q,\alpha) \leq \frac{1}{2^{\lfloor \frac{n+1}{2} \rfloor}-1}.\]

    \begin{proof}
        We define an unstable representation $X$ of $Q$ realizing the bound. After contracting all interior vertices of $Q$ we obtain an $n$-star quiver $\overline{Q}$. Let $Y$ be the unstable representation of $\overline{Q}$ defined in \Cref{weightmargin_star_quiver} and denote its dimension vector $\beta$. For each leaf $k\in [n]$ let $X(k):=Y(k)$ and let $X(v):=Y(0)$ for each $v\in Q_0$ which is not a leaf. For each leaf $k\in [n]$ denote its adjacent arrow as $k$ and let $X_k := Y_k$. We note that by the definition of $X$ we have
        \[ \mu_k(X) = \mu_k(Y)= 0.\]
        We define the rest of the matrices inductively. The restriction of $Q$ to all non-leaf vertices is again a tree quiver $Q'$. Let $v\in Q'_0$ be a leaf with adjacent arrow $b$. Now let \[M:= \sum_{\substack{a\in Q_1\backslash{Q_1'}\\ ha=v}} X_aX_a^* - \sum_{\substack{a\in Q_1\backslash{Q_1'}\\ ta=v}} X_a^*X_a \]
        This matrix is diagonal with real entries, so there exists a unique $\alpha_a\in \bR$, such that $M+ \alpha_a I$ is a diagonal matrix with non-negative entries and $\det(M+\alpha_a I)=0$. In case $ta=v$ we let $X_a :=\sqrt{M+\alpha_a I}$ and see
        \[ \mu_v(X) = p\Big(\sum_{\substack{a\in Q_1\backslash{Q_1'}\\ ha=v}} X_aX_a^* - \sum_{\substack{a\in Q_1\backslash{Q_1'}\\ ta=v}} X_a^*X_a - \sqrt{M+\alpha_a I}^2\Big) = p(M- (M+\alpha_a)) = p(-\alpha_a I) = 0.\]
        
        In case $ha=v$ there exists a unique $\alpha_a\in \bR$, such that $M+ \alpha_a I$ is a diagonal matrix with non-positive entries and $\det(M+\alpha_a I)=0$ and we define $X_a:=\sqrt{-(M+\alpha_a I)}$. In this case we have
        \[ \mu_v(X) = p\Big(\sum_{\substack{a\in Q_1\backslash{Q_1'}\\ ha=v}} X_aX_a^* - \sum_{\substack{a\in Q_1\backslash{Q_1'}\\ ta=v}} X_a^*X_a + \sqrt{-(M+\alpha_a I)}^2\Big) = p(M- (M+\alpha_a)) = p(-\alpha_a I) = 0.\]
        Now we remove $v$ from $Q'$ and proceed inductively until $Q'$ is a single vertex $u$. At each step in the construction we have $\sum_{v\in Q'} \mu_v(X) = \mu_0(Y)$. This implies that we have
        \[\mu_u(X) = \mu_0(Y).\]
        Thus we obtain
        \[ \|\mu(X)\| \leq \frac{\|\mu_u(X)\|}{\sum_{k=1}^n \tr(X_k^*X_k)} = \gamma(p({0,1}^{n+1})) \leq \frac{1}{2^{\lfloor \frac{n+1}{2} \rfloor} - 1}.\]
        
        It remains to show that $X$ is unstable. By the proof of \Cref{weightmargin_star_quiver} there exists a one parameter subgroup of diagonal matrices $M(t) \in \ST_{\beta}$, such that $\lim_{t\to 0}M(t)\cdot Y = 0$. We let $M(t)$ act on $X$ by acting with $M_k(t)$ at each leaf $k\in [n]$ and with $M_0(t)$ at each non-leaf vertex of $Q$. Since each $X_a$ is a diagonal matrix for $a\in Q_1\backslash{[n]}$, we obtain $\lim_{t\to 0} M(t) \cdot X = X'$, where $X'_k = 0$ for $k\in [n]$ and $X'_a=X_a$ for each $a\in Q_1\backslash{[n]}$. But now $X'$ is a representation of a tree quiver $Q'$ with uniform dimension vector, such that $\det(X'_a)=0$ for each $a\in Q'_1$. Such a representation is unstable by \Cref{nullcone_tree_uniform}.
    \end{proof}
\end{theorem}

\section{Weightmargin of arbitrary quivers}

In the previous section we showed that tree quivers have exponentially small gap (for certain dimension vectors) in the number of leaves. In this section we extend this result to connected quivers with $n+1$ vertices and an $n$-fold copy of an arrow next to a leaf. For example the following quiver
\[\begin{tikzcd}
	\bullet & \bullet & \bullet & \bullet
	\arrow[from=1-1, to=1-2]
	\arrow[from=1-2, to=1-3]
	\arrow[from=1-3, to=1-4]
	\arrow[shift left=3, from=1-3, to=1-4]
	\arrow[shift right=3, from=1-3, to=1-4]
\end{tikzcd}\]
Since the weightmargin is upper bounded by the gap and doesn't see arrow multiplicity, we obtain as a corollary, that any connected quiver with $n+1$ vertices has a dimension vector (bounded by $n^2$), such that the weightmargin is exponentially small in $n$.

\begin{theorem}
\label{gap_arrow_multiplicity}
    Let $Q'$ be a tree quiver with $n+1$ vertices and let $Q$ be a quiver obtained from $Q'$ by replacing an arrow adjacent to a leaf with $n$ copies. Then there exists a dimension vector $\alpha$ with $\alpha_{\max} \leq n^2$, such that
    \[ \gamma_{\SL}(Q,\alpha) \leq \frac{1}{2^{\lfloor \frac{n+1}{2}\rfloor}-1} \]

    \begin{proof}
        W.l.o.g. assume that $n$ is odd. Let $v$ be the leaf of $Q'$ with adjacent arrow $b$, such that $b$ is replaced by $n$ copies $b_1,\ldots, b_n$ in $Q$. Assume w.l.o.g. that $hb=v$ (otherwise pass to the dual quiver $Q^*$). Let $\Tilde{Q}$ be the quiver obtained from $Q$ by changing the orientation of all arrows such that they point towards $v$, we write $\Tilde{h}, \Tilde{t}$ for the head and tail of arrows in $\Tilde{Q}$. We inductively define a function $f:Q_0\backslash{v}\to P([n])$ and a function $g:Q_0\backslash{v}\to [n]$. Let $w:=tb$ be the unique vertex adjacent to $v$ in $Q'$. Let $f(w) := [n]$. Now $w$ has some incoming arrows $a_1,\ldots, a_L \in \Tilde{Q}_1$ whose tails we denote $w_k:=\Tilde{t}a_k$. Assume the $a_k$ are sorted, such that $ha_k = w$ for all $1\leq k \leq l$ and $ta_k = w$ for all $l< k \leq L$. Each $w_k$ determines a subquiver (the connected component of $\Tilde{Q}\backslash{a_k}$ containing $w_k$) with $n_k$ vertices. We see $|f(w)|=n = 1+\sum_{k=1}^L n_k$. Now we assign $f(w_1):= [1, n_1]$, $f(w_k):=[1+\sum_{i=1}^{k-1} n_i, \sum_{i=1}^{k} n_i]$ for $1\leq i\leq l$ and $f(w_k):=[2+\sum_{i=1}^{k-1} n_i, 1+\sum_{i=1}^{k} n_i]$ for $l< k\leq L$. Finally we let $g(w):=1+\sum_{k=1}^l n_k$ and we see that
        \[ f(w)=\{ g(w) \} \cup \bigcup_{k=1}^L f(w_k).\]
        Now we proceed to define $f(w_k)$ and $g(w_k)$ for $1\leq k\leq L$. This way we inductively define $f$ and $g$.
        
        Let $S$ be the equioriented n-star quiver and let $\beta$ and $Y\in \Rep_\beta(S)$ be the dimension vector and representation constructed in \Cref{weightmargin_star_quiver}. From the construction of $Y$ we know that $Y_k^*Y_k = \lambda_k I_{\beta(k)}$ for some $\lambda_k\geq 0$ and for each $k\in [n]$ and w.l.o.g. we assume that $\lambda_1 \geq \lambda_{2} \geq \ldots \geq \lambda_n$. We define a representation $X\in \Rep_{\alpha}(Q)$. Let $X(v):=Y(0)$ and $X(u):=\bigoplus_{k\in f(u)} Y(k)$ for each $w\in Q_0\backslash{\{v\}}$. For each $k\in[n]$ we let $X_{b_k} := Y_k \circ \pi_{Y(k)}$, where $\pi_{Y(k)} : \bigoplus_{l=1}^{n} Y(l) \to Y(k)$ is the projection. We see
        \[ X_{b_k}^*X_{b_k} = \lambda_k \iota_{Y(k)} \circ\pi_{Y(k)}\]
        \[ \|\mu_v(X)\|^2 = \|\sum_{k=1}^n p(X_{b_k}X_{b_k}^*)\|^2 = \|\sum_{k=1}^n p(Y_kY_k^*)\|^2 \leq \frac{1}{2^{\lfloor\frac{n}{2}\rfloor} - 1} \]
        Now let $a\in Q_1\backslash{\{b_1,\ldots, b_n\}}$ be some other arrow and let $w:=\Tilde{h}a$ and $u:=\Tilde{t}a$. We let
        \[ X_a := \sum_{k\in f(u)} \sqrt{|\lambda_k-\lambda_{g(w)}|}\iota_{Y(k)} \circ \pi_{Y(k)}.\]
        Next we show, that $\mu_u(X)=0$ for all $u\in Q_0\backslash{\{v\}}$. As above let $w\in Q_0$ be the vertex adjacent to $v$ and denote its incoming arrows in $\Tilde{Q}$ as $a_1,\ldots, a_L$. We have
        \begin{align*}
            \mu_w(X) &= p\Big(\sum_{k=1}^l X_{a_k}X_{a_k}^* - \sum_{k=l+1}^L X_{a_k}^*X_{a_k} -  \sum_{k=1}^n X_{b_k}^*X_{b_k} \Big)\\
            &= p\Big(\sum_{\substack{k\in f(w)\\ k<g(w)}} (\lambda_k - \lambda_{g(w)}) \id_{Y(k)} - \sum_{\substack{k\in f(w)\\ k>g(w)}} (\lambda_{g(w)} - \lambda_{k}) \id_{Y(k)} - \sum_{k=1}^n \lambda_k \id_{Y(k)}\Big)\\
            &= p\Big(\sum_{k\in [n]\backslash{g(w)}} (\lambda_{k} - \lambda_{g(w)}) \id_{Y(k)} - \sum_{k=1}^n \lambda_k \id_{Y(k)}\Big)\\
            &= p\Big( - \lambda_{g(w)} \sum_{k=1}^n \id_{Y(k)}\Big) = 0
        \end{align*}
        For any other vertex $u$ we use a similar argument. Let $a\in \Tilde{Q}$ be its unique outgoing arrow and denote its incoming arrows as $a_1,\ldots, a_L \in \Tilde{Q}$. Let $w:=\Tilde{h}a$. In case $ta = u$ let 
        \[R:= -X_a^*X_a = -\sum_{k\in f(u)} |\lambda_k-\lambda_{g(w)}| \id_{Y(k)} = \sum_{k\in f(u)} (\lambda_{g(w)}-\lambda_{k}) \id_{Y(k)}.\]
        In case $ha = u$ let
        \[ R:=X_aX_a^* = \sum_{k\in f(u)} |\lambda_k-\lambda_{g(w)}| \id_{Y(k)} = \sum_{k\in f(u)} (\lambda_{g(w)}-\lambda_{k}) \id_{Y(k)}. \]
        Now we can calculate
        \begin{align*}
            \mu_u(X) &= p\Big(\sum_{k=1}^l X_{a_k}X_{a_k}^* - \sum_{k=l+1}^L X_{a_k}^*X_{a_k} +  R \Big)\\
            &= p\Big(\sum_{k\in f(u)\backslash{g(u)}} (\lambda_{k} - \lambda_{g(u)}) \id_{Y(k)} + \sum_{k\in f(u)} (\lambda_{g(w)}-\lambda_{k}) \id_{Y(k)} \Big)\\
            &=  p\Big((\lambda_{g(w)}-\lambda_{g(u)})\sum_{k\in f(u)}\id_{Y(k)}\Big) = 0
        \end{align*}
        Thus we obtain
        \[ \|\mu(X)\| \leq \frac{\|\mu_v(X)\|}{\sum_{k=1}^n \tr(Y_k^*Y_k)} \leq \frac{1}{2^{\lfloor\frac{n}{2}\rfloor} - 1}.\]
        It remains to show, that $X$ is unstable. In \Cref{weightmargin_star_quiver} we constructed a one parameter subgroup $M:\bC^* \to \ST_{\beta}$ killing $Y$. We extend this to a one parameter subgroup $\overline{M}:\bC^* \to \ST_{\alpha}$, where $\alpha = \dim(X)$. For each vertex $u\in Q_0\backslash{\{v\}}$, we have $X(u)= \bigoplus_{k\in f(u)} Y(k)$ and we let
        \[ \overline{M}_u(t) := \diag((M_k(t))_{k\in f(u)}).\]
        For $v\in Q_0$ we let $\overline{M}_u(t) :=M_0(t)$. For each arrow $b_1,\ldots, b_n$ we see
        \[ \lim_{t\to 0} (\overline{M}(t) \cdot X)_{b_k} = \lim_{t\to 0} M_0(t) Y_k M_k(t)^{-1} \pi_{Y(k)} = 0.\]
        For any other arrow $a\in Q_1\backslash{\{b_1,\ldots b_n\}}$ let $(w,u):=(\Tilde{h}a, \Tilde{t}a)$ and we see
        \[\lim_{t\to 0} (\overline{M}(t) \cdot X)_{a} = \lim_{t\to 0} \sum_{k\in f(u)} \sqrt{|\lambda_k-\lambda_{g(w)}|} \iota_{Y(k)}M_k(t)M_k(t)^{-1}\pi_{Y(k)} = X_a.\]
        Denote $X' := \lim_{t\to 0} \overline{M}(t) \cdot X$. This is a representation of a tree quiver which we will show to be unstable. Let $w:= tb$ be the vertex adjacent to $v$. We act at $w$ with the following matrix
        \[ \diag ((tI_{\beta(k)})_{k<g(w)}, t^s I_{\beta(g(w))}, (t^{-1}I_{\beta(k)})_{k>g(w)}), \]
        where $s:= \frac{\sum_{k>g(w)} \beta(k)-\sum_{k<g(w)} \beta(k)}{\beta(g(w))}$, such that the matrix has determinant one. Letting $t$ tend to zero kills all arrows adjacent to $w$ and we can proceed by induction sending $X'$ to zero.
    \end{proof}

    \begin{corollary}
        \label{weight_margin_connected quivers}
        Let $Q$ be a connected quiver with $n+1$ vertices. Then there exists a dimension vector $\alpha$ with $\alpha_{\max} \leq n^2$, such that
    \[ \gamma_{\ST}(Q,\alpha) \leq \frac{1}{2^{\lfloor \frac{n+1}{2}\rfloor}-1}. \]
    \begin{proof}
        W.l.o.g assume that $Q$ is a tree and let $Q'$ be the quiver obtained from $Q$ by replacing some arrow which is adjacent to a leaf with $n$ copies. Since the weightmargin does not see arrow multiplicity we have
        \[ \gamma_{\ST}(Q,\alpha) = \gamma_{\ST}(Q',\alpha) \leq \gamma_{\SL}(Q',\alpha) \leq \frac{1}{2^{\lfloor \frac{n+1}{2}\rfloor}-1},\]
        where $\alpha$ and the second inequality comes from \Cref{gap_arrow_multiplicity}. 
    \end{proof}
    \end{corollary}
\end{theorem}

%% file: real_minimization_problems.tex
\section{Real minimization problems}

In this appendix we solve some minimization problems over the reals, which were used in the proofs of \Cref{lem:graph_component_inequality} and \Cref{prop:component_lower_bound}.

\begin{lemma}
\label{minlemma4}
    Let $x,m\in \bR^n$, such that $m_i> 0$ for all $i\in[n]$. Then the following equation holds
    \[ \min_{x\in \bR^n} \frac{\sum_{i=1}^n m_i x_i^2}{(\sum x_i)^2} = \frac{1}{\sum_{i=1}^n m_i^{-1}}. \]
    \begin{proof}
    Using the Cauchy-Schwarz inequality in the second step we see
    \[ \left(\sum_{i=1}^n x_i\right)^2 = \left(\sum_{i=1}^n m_i^{-\frac{1}{2}} \cdot m_i^{\frac{1}{2}} x_i\right)^2 \leq \left(\sum_{i=1}^n m_i^{-1}\right) \cdot \left(\sum_{i=1}^n m_ix_i^2\right).\]
    This yields
    \[ \min_{x\in \bR^n} \frac{\sum_{i=1}^n m_i x_i^2}{(\sum x_i)^2} \geq  \frac{1}{\sum_{i=1}^n m_i^{-1}},\]
    and setting $x_i := m_i^{-1}$ we see that this minimum is indeed attained.
    \end{proof}
\end{lemma}

\begin{lemma}
\label{lem:minlemma2}
    Let $x_1,\ldots, x_n\geq 0$ and consider the function
    \[ f(x) = \frac{x_n^2 + \sum_{i=1}^{n-1} (x_{i+1}-x_i)^2}{(\sum_{i=1}^n x_i)^2}.\]
    Then we have \[\min f(x) = \frac{6}{n(n+1)(2n+1)} \geq 3(n+1)^{-3}.\]

    \begin{proof}
        Substituting $y_n := n x_n$ and $y_k := k(x_k - x_{k+1})$ for $k \leq n-1$ yields using \Cref{minlemma4}
        \[ \min_x f(x) = \min_y \frac{\sum_{k=1}^{n} k^{-2} y_k^2}{(\sum_{i=1}^n y_i)^2} = \left( \sum_{k=1}^n k^2 \right)^{-1} = \frac{6}{n(n+1)(2n+1)}.\]
    \end{proof}
\end{lemma}

\begin{lemma}
\label{minlemma_some_term_vanishes}
    Let $x_1,\ldots, x_{n}\in \bR$, such that $\sum_{i=1}^n x_i > 0$ and some $x_k$ is not positive. Consider the function
    \[ f(x) = \frac{\sum_{i=1}^{n-1} (x_{i+1}-x_i)^2}{(\sum_{i=1}^n x_i)^2}.\]
    Then we have \[\min f(x) = \frac{6}{(n-1)n(2n-1)} \geq 3n^{-3}.\]
    \begin{proof}
        Replacing each $x_i$ with $x_i-\min\{x_j \;|\; j\in [n]\}$ can only decrease $f(x)$, so we may assume that all $x_i\geq 0$. Let $k\in [n]$ be the smallest index, such that $x_k=0$. In case $k\in \{1,n\}$ we are done by \Cref{lem:minlemma2}. Otherwise we assume w.l.o.g. that $k\leq \frac{n+1}{2}$. Let $y:=\min\{x_i \;|\; 1\leq i <k\}$. Now we replace $x_i$ with $x_i-y$ for $1\leq i<k$ and with $x_i+y$ for $k\leq i\leq n$. This does not increase the numerator of $f$ and does increase the denominator (since more than half of the components of $x$ get increased), thus we can inductively reduce to the case $k=1$ and we are done.
    \end{proof}
\end{lemma}

\begin{lemma}
\label{minlemma3}
    Let $c\geq 0$ and $x_1,\ldots, x_n\geq 0$ as well as $\lambda_1,\ldots, \lambda_{n-1}\geq 0$ such that $\sum_{i=1}^{n} \lambda_i = c$. Consider the function
    \[f(x,\lambda) = \frac{\sum_{i=1}^{n-1} (x_i-x_{i+1})^2 + \lambda_i x_ix_{i+1}}{(\sum_{i=1}^n x_i)^2} .\]
    We have $f(x, \lambda) \geq  \min (\frac{1}{2}cn^{-2}, \frac{3}{2} n^{-3})$.

    \begin{proof}
        By homogeneity of the function we can restrict the domain to the compact set \\$\{ \sum_{i=1}^n x_i = 1 \;|\; x_i\geq 0 \; \forall i\in [n]\}$ and see that (for each choice of $\lambda$) the minimum must be attained at some point $x$ (that depends on $\lambda$).\\
        Since the function is linear in $\lambda$ we can assume $\lambda_j= c$ for some $j\in [n-1]$ and $\lambda_i = 0$ for $j\neq i$. Next we will show, that we can assume $j\in \{1,n-1\}$. Assume that $j \neq 1,n-1$. Then at the minimum we have $x_{i} \geq x_{i+1}$ for $i<j$. Otherwise there would be some minimal $i\leq j-1$ such that $x_{i}<x_{i+1}$ and we could replace $x_k$ with $x_k+ 2 (x_{i+1}-x_i)$ for $k\leq i$ which would not change the value of the numerator of $f(x,\lambda)$ while increasing the denominator, a contradiction to $x$ attaining the minimum of $f(\cdot, \lambda)$. In a similar way we obtain $x_i\leq x_{i+1}$ for $i\geq j+1$. Now assume w.l.o.g. that $x_1\leq x_n$. Then we "shift" all variables to the left by replacing $x_i$ with $x_{i+1}$ and $x_n$ with $x_n + (x_1-x_2)$ and $j$ with $j-1$. One easily checks that this doesn't change the numerator of $f(x,\lambda)$ while not decreasing the denominator. After this process $x_1\leq x_n$ still holds so we can repeat the process until $j=1$. In case $x_n>x_1$ we can use a similar "right shifting process" to obtain $j=n-1$.\\ Thus we can assume $j\in \{1,n-1\}$ and by symmetry $j=n-1$, so $f$ is of the form
        \[ f(x,\lambda) = \frac{\sum_{i=1}^{n-1} (x_i-x_{i+1})^2 + c x_{n-1}x_{n}}{(\sum_{i=1}^n x_i)^2}.\]
        Since $x$ is decreasing, we can lower bound $x_{n-1}x_n \geq x_n^2$ and substitute $y_n := n x_n$ and $y_i := i(x_i-x_{i+1})$ for $i \leq n-1$ to obtain 
        \begin{align*}
            f(x, \lambda) &\geq \frac{c x_{n}^2 + \sum_{i=1}^{n-1} (x_i-x_{i+1})^2}{(\sum_{i=1}^n x_i)^2} \\
            &= \frac{c n^{-2} y_{n}^2 + \sum_{i=1}^{n-1} i^{-2} y_i^2}{(\sum_{i=1}^n y_i)^2}\\
            &\geq (n^2 c^{-1} + \sum_{i=1}^{n-1} i^2)^{-1}\\
            &= (n^2 c^{-1} + 6^{-1}n(n-1)(2n-1))^{-1}\\
            &\geq (n^2 c^{-1} + 3^{-1}n^3)^{-1}\\
            &\geq \min\left(\frac{1}{2} cn^{-2}, \frac{3}{2} n^{-3}\right).
        \end{align*}
    \end{proof}
\end{lemma}

The next Lemma is slightly weaker and immediately follows from the previous one.

\JM{TODO : cut this lemma}
\begin{lemma}
    \label{lem:minlemma5}
    Let $c\geq 0$ and $x_1,\ldots, x_n\geq 0$ as well as $\lambda_1,\ldots, \lambda_{n}\geq 0$ such that $\sum_{i=1}^{n} \lambda_i = c$. Consider the function
    \[f(x,\lambda) = \frac{ (x_n-x_1)^2 + \lambda_n x_nx_1 + \sum_{i=1}^{n-1} (x_i-x_{i+1})^2 + \lambda_i x_ix_{i+1}}{(\sum_{i=1}^n x_i)^2} .\]
    We have $f(x, \lambda) \geq  \min (\frac{1}{2}cn^{-2}, \frac{3}{2} n^{-3})$.

    \begin{proof}
        Since the function is linear in $\lambda$, we can assume $\lambda_i=c$ for some index $i\in[n]$ and by a cyclic permutation of the variables we can further assume $i=n-1$. Using \Cref{minlemma3} we obtain 
        \[f(x,\lambda) \geq \frac{\sum_{i=1}^{n-1} (x_i-x_{i+1})^2 + \lambda_i x_ix_{i+1}}{(\sum_{i=1}^n x_i)^2} \geq \min (\frac{1}{2}cn^{-2}, \frac{3}{2} n^{-3}).\]
    \end{proof}
\end{lemma}

\begin{lemma}
\label{lem:sum_of_cubes}
    Let $a,b\geq 0$ and $x_1,\ldots, x_n \in [0,a]$, such that $\sum_{i=1}^n x_n\leq ab$. Then we have
    \[ \sum_{i=1}^n x_i^3 \leq ba^3.\]
    \begin{proof}
        We immediately see
        \[ \sum_{i=1}^n \left(\frac{x_i}{a}\right)^3 \leq \sum_{i=1}^n \frac{x_i}{a} \leq b.\]
        The result follows by multiplying with $a^3$.
    \end{proof}
\end{lemma}